\documentclass{article}
\usepackage[utf8]{inputenc}
\usepackage[english]{babel}
\usepackage{hyperref}
\usepackage{amsmath}
\usepackage{amsthm}
\usepackage{amssymb}
\usepackage{graphicx}
\usepackage{subcaption}
\usepackage{comment}
\usepackage{todonotes}
\usepackage{algorithm}
\usepackage{algpseudocode}
\usepackage{authblk}
\usepackage{mathtools}
\usepackage{fullpage}
\usepackage{float}
\providecommand{\keywords}[1]{\textbf{\textit{Keywords---}} #1}
\DeclarePairedDelimiter{\diagfences}{(}{)}
\newcommand{\diag}{\operatorname{diag}\diagfences}
\DeclareMathOperator*{\argmax}{arg\,max}
\DeclareMathOperator\supp{supp}
\title{Super-resolution of generalized spikes and spectra of confluent Vandermonde matrices\thanks{This research was supported by Israel Science Foundation grant 1792/20 and by Lower Saxony - Israel collaboration grant from the Volkswagen Foundation.}}

\author[1]{Dmitry Batenkov} \affil[1]{Department of Applied
	Mathematics, Tel-Aviv University, Tel-Aviv, Israel.}
\author[1]{Nuha Diab}

\affil[ ]{\textit{dbatenkov@tauex.tau.ac.il, nuhadiab@tauex.tau.ac.il}}

\date{\today}

\usepackage[backend=biber,style=ieee]{biblatex}
\usepackage{graphicx}

\newtheorem{definition}{Definition}[section]
\newtheorem{theorem}{Theorem}[section]
\newtheorem{corollary}{Corollary}[section]
\newtheorem{lemma}{Lemma}[section]
\newtheorem{proposition}{Proposition}[section]

\usepackage{mathtools}

\begin{document}
	\maketitle
	
	\begin{abstract}
		We study the problem of super-resolution of a linear combination of Dirac distributions and their derivatives on a one-dimensional circle from noisy Fourier measurements. Following numerous recent works on the subject, we consider the geometric setting of ``partial clustering'', when some Diracs can be separated much below the Rayleigh limit. Under this assumption, we prove sharp asymptotic bounds for the smallest singular value of a corresponding rectangular confluent Vandermonde matrix with nodes on the unit circle. As a consequence, we derive matching lower and upper min-max error bounds for the above super-resolution problem, under the additional assumption of nodes belonging to a fixed grid.
	\end{abstract}
	

	\section{Introduction}
	
	\subsection{Background}
	The problem of computional super-resolution (SR) is to recover
	the fine details of an unknown object from inaccurate
	measurments of inherently low resolution
	\cite{donoho1992a}. In recent years, there is much intrest in
	the problem of reconstructing a signal modelled by a linear
	combination of Dirac $\delta-$distributions
	(e.g. \cite{batenkov2021, batenkov2021a, Vandermonde,
		Li&Liao2020,candes2014,demanet2015,hockmann2021,liu2021,petz2021,cuyt2018}
	and references therein):
	\begin{equation}\label{simpleModel}
		\mu(x) = \sum_{j=1}^{s}a_j\delta_{\xi_j}  \quad a_j\in\mathbb{C},\; \delta_{\xi_j}=\delta(x-\xi_j), \ \xi_j \in (-\pi,\pi]
	\end{equation}
	from noisy and bandlimited Fourier measurements:
	\begin{equation}\label{eq:noisy-data}
		y_k := \hat{\mu}(k)+\eta_k, \quad  \hat{\mu}(k)=\langle \mu, e^{-ikx} \rangle , \quad k=0,1,...,M, \; |\eta_k|\leq\varepsilon.
	\end{equation}
	
	For the model \eqref{simpleModel} we have $\hat{\mu}(k)= \sum_{j=1}^{s}a_je^{ik\xi_j}$, and therefore the measurement vector $y=\{y_k\}_{k=0}^M $ can be expressed as
	\begin{equation}
		\label{eq:vandermonde-measurements}
		y = Va + \eta \in \mathbb{C}^{M+1}
	\end{equation}
	where $V$ is the $(M+1) \times s$ Vandermonde matrix with the
	nodes on the unit circle:
	\[ V := [e^{ik\xi_j}]_{k=0,...,M}^{j=1,...,s}.\] In order to
	describe the stability of this inverse problem, suppose that
	the nodes $\xi_j$ belong to a grid of step size $\Delta$\, and
	define the super-resolution factor (SRF) as
	$\frac{1}{(M\Delta)}$.  Suppose that at most
	$\ell \leqslant s$ nodes form a "cluster" of size $O(\Delta)$
	(to be rigorously defined below). In the ''super-resolution
	regime'' $SRF\gg 1$ \cite{Vandermonde,Li&Liao2020} showed that
	$\sigma_{\min}(V)$ scales like ${SRF}^{1-\ell}$ and
	consequently the worst-case reconstruction error rate of the
	coefficients of $\mu$ as in \eqref{simpleModel} from noisy
	measurements \eqref{eq:noisy-data} is of the order
	${SRF}^{2\ell -1}\varepsilon$. Despite the great amount of
	research devoted to the subject, there is currently no known
	tractable algorithm which provably achieves
	these min-max bounds for all signals of interest
	\cite{batenkov2021a}.
	
	\subsection{Our contributions}
	
	In this work we extend the methods and results of
	\cite{Vandermonde,Li&Liao2020} to the model
	\begin{equation}\label{generalizedModel}
		\mu = \sum_{j=1}^{s}a_j\delta_{\xi_j}+b_j\delta'_{\xi_j},
	\end{equation}
	
	where $\delta^{'}$ is the distributional derivative
	of the Dirac delta. The Vandermonde matrix $V$ in
	\eqref{eq:vandermonde-measurements} is replaced by the
	so-called \emph{confluent Vandermonde matrix} $U$, which is
	defined (up to normalization) as:
	\[ U := [e^{ik\xi_j} \
	ke^{i(k-1)\xi_j}]_{k=0,...,M}^{j=1,...,s}.\] Under the
	partial clustering assumptions, in Theorem \ref{mainTheorem} and \ref{UpperTheorem} we prove a sharp lower and upper bounds
	for the smallest singular value of $U$ in the super-resolution
	regime, and show that it scales like ${SRF}^{1-2\ell}$. These
	bounds are proved by extending the decimation approach from
	\cite{Vandermonde} for the lower bound on $\sigma_{\min}(U)$,
	and by extending the finite difference approximation approach
	from \cite{Li&Liao2020} for the upper bound, further generalizing it to any node vector $\boldsymbol{\xi}$ satisfying the clustering assumptions. In addition, our
	proof technique for bounding the remainder part in the upper
	bound of the smallest singular value can be applied to gain a
	slight improvement in Proposition 2.10 in \cite{Li&Liao2020}
	by relaxing the conditions on $M,\Delta$.
	
	As a consequence, in Theorem \ref{theorem:min-max-error} we also obtain sharp min-max bounds of order
	${SRF}^{4\ell -1}\varepsilon$ for the problem of sparse
	super-resolution of signals \eqref{generalizedModel} on a grid
	by extending the corresponding technique from
	\cite{Vandermonde}.
	
	Also, we show numerically that the well-known ESPRIT method for exponential fitting (appropriately extended to handle higher multiplicities) is optimal, meaning that it attains the min-max error bounds we established in Theorem \ref{theorem:min-max-error} for the recovered parameters of the signal \eqref{generalizedModel}.
	
	In relation to prior work on the subject, in
	\cite{batenkov2013c} the authors give a stability estimate for the more general model 
	\eqref{extendedmodel} with arbitrary fixed $n$, however
	assuming that the number of measurements $N+1$ equals the
	number of unknowns. Evaluating their estimate for our model
	and notation, their bound is of order
	${\Delta}^{4\ell -1}{\varepsilon}^{\frac{1}{2}}$, while ours in
	the same case is ${\Delta}^{4\ell -1}{\varepsilon}$.
	In contrast, \cite{batenkov2018} established the bound in the
	super-resolution setting of a single cluster (and off-grid
	nodes) to be of order ${SRF}^{4\ell}\varepsilon$, while we
	derive the min-max rate ${SRF}^{4\ell-1}\varepsilon$.
	
	\subsection{Discussion}
	
	Naturally, our results and techniques pave the way to analyzing the
	general model
	\begin{equation}\label{extendedmodel}
		\mu = \sum_{j=1}^{s}\sum_{l=0}^{n}a_{j,l}\delta_{\xi_j}^{(l)}
	\end{equation}
	in the clustered super-resolution regime. The applications of
	this model include modern sampling theory beyond the Nyquist
	rate, algebraic signal recovery, interpolation and multi-exponential
	analysis, to name a few (see \cite{Badeau2006,batenkov2012c,
		batenkov2013c,batenkov2015a, batenkov2018, Sidi1982, Badeau2008} and references
	therein). At the same time, we believe that several recent
	developments on the basic model \eqref{simpleModel} can be
	utilized to the more general setting, as follows.
	\begin{itemize}
		\item Recently, \cite{batenkov2021} succeeded to establish
		sharp bounds for the entire spectrum of $V$ without
		requiring the entire node set to be contained in a small
		interval of length $\frac{\pi}{s^2}$. We believe that similar
		techniques could be applied to \eqref{generalizedModel}
		in order to eliminate the above restriction.
		\item Optimal scaling of the constants in the above bounds for the spectrum of $V$  using harmonic analysis techniques has been investigated in \cite{batenkov_single-exponential_2021}, and it would be interesting to extend these methods to \eqref{extendedmodel}.
		\item While we obtain min-max rates for nodes on a grid, we
		expect to get similar rates for the ''off-grid'' model as in
		\cite{batenkov2021a}, where the node locations can be any real
		number. Furthermore, it should be possible to establish
		component-wise bounds for the coefficients of different
		orders and for the nodes themselves, as done in
		\cite{batenkov2015a, batenkov2013c, batenkov2018} for the
		more restrictive geometric settings of the
		problem. 
	\end{itemize}
	
	Going back to the model \eqref{simpleModel}, confluent
	Vandermonde matrices appear naturally in the perturbation
	analysis of the nonlinear least squares problems
	\cite{batenkov2021a,batenkov2018} for exponential fitting, and we expect our methods
	to be applicable in this context as well.
	
	The paper is organized as follows. In section
	\ref{sec:preliminaries} we establish some notation. In section
	\ref{sec:main-results} we formulate the main results, which
	are proved in section \ref{sec:proofs}. Finally, in section \ref{sec:numerics}
	we present numerical evidence confirming our bounds.
	
	\section{Preliminaries}\label{sec:preliminaries}
	\subsection{Notation}
	
	\begin{definition}
		\label{reconfluentVan}
		For $N \in 	\mathbb{N}$ and a vector 
		$\boldsymbol{\xi} = (\xi_1, . . . , \xi_s)$ of pairwise distinct real
		nodes $\xi_j \in ( - \pi , \pi ]$, we define the rectangular $(2N + 1) \times 2s$ {\bf confluent Vandermonde matrix}
		$U_N (\boldsymbol{\xi})$ as
		\[ U_N (\boldsymbol{\xi}) := \frac{1}{\sqrt{2N}}\big[z^k_j \quad k{z_j}^{k-1}\big]^{j=1,\dots,s} _{k=0,\dots, 2N} \]
		s.t. $z_j = \exp(i\xi_j)$.
	\end{definition}
	
	The main subject of the paper is the scaling of the smallest singular value of $U_N$ when some of the nodes of $\boldsymbol{\xi}$ nearly collide (become very close to each other).
	
	\begin{definition}[wraparound distance]
		For $t \in \mathbb{R}$, we denote
		\[\|t\|_{\tilde{\mathbb{T}}} := \big|\operatorname{Arg}\exp(it)\big| = \big|t \ \mod \ ( -\pi , \pi ]\big|,\]
		where $\operatorname{Arg}(z)$ is the principal value of the argument of $z \in \mathbb{C}\setminus\{0\}$, taking values in $( -\pi , \pi ]$.
	\end{definition}
	
	\begin{definition}[minimal separation]
		Given a vector of s distinct nodes $\boldsymbol{x} := (x_1,\dots,x_s)$ with $x_j \in ( -\pi , \pi ]$, we define the minimal separation (in wraparound sense) as
		\[\Delta := \Delta(\boldsymbol{x}) = \min_{i \neq j}\|x_i - x_j\|_{\tilde{\mathbb{T}}} \ . \]
	\end{definition}
	
	\begin{definition}\label{cluster_confg}
		The node vector $\boldsymbol{x} = (x_1,\dots,x_s) \subset \big( -\frac{\pi}{2} , \frac{\pi}{2} \big]$ is said to form a $(\Delta,\rho,s,\ell,\tau)$- clustered configuration for some $\Delta > 0, 2 \leq \ell \leq s, \ell-1 \leq \tau \leq \frac{\pi}{\Delta},$ and $\rho \geq 0$ if for each $x_j$ there exist at most $\ell$ distinct nodes
		\[ \boldsymbol{x}^{(j)} = \{x_{j,k}\}_{k=1,\dots, r_j} \subset \boldsymbol{x},\quad 1\leq r_j \leq \ell,\quad x_{j,1} \equiv x_j,\]
		such that the following conditions are satisfied:
		\begin{enumerate}
			\item For any $y \in \boldsymbol{x}^{(j)}\setminus \{x_j\}$, we have
			\[\Delta \leq \|y - x_j\|_{\tilde{\mathbb{T}}} \leq \tau\Delta. \]
			\item For any $y \in \boldsymbol{x} \setminus \boldsymbol{x}^{(j)}$, we have
			\[\|y - x_j\|_{\tilde{\mathbb{T}}} \geq \rho.\]
		\end{enumerate} 
	\end{definition}
	
	\begin{definition}\label{grid}
		For $\Delta > 0$ let $M=\big\lfloor\frac{\pi}{2\Delta}\big\rfloor$ and denote by $\mathcal{T}_{\Delta}$ the discrete grid 
		\[ \mathcal{T}_{\Delta} := \{k\Delta,\quad k=-M,\dots,M\} \subset \bigg[-\frac{\pi}{2},\frac{\pi}{2}\bigg] \ . \]
		Further define $G:= G(\Delta) =\big|\mathcal{T}_{\Delta}\big|=2M+1$.
	\end{definition}

	\begin{definition}
		For $\Delta,\rho,s,\ell,\tau$ as in Definition \ref{cluster_confg}, let $\mathcal{R}:=\mathcal{R}(\Delta,\rho,s,\ell,\tau)$ be the set of point distributions of the form $\mu=\sum_{j=1}^{s}a_j\delta_{t_j}+b_j\delta'_{t_j}$ where $t_j \in \mathcal{T}_{\Delta}$ and $a_j, b_j \in \mathbb{C}$ for all $j=1,\dots,s$, while $\boldsymbol{t}=(t_1,\dots,t_s)$ forms a $(\Delta,\rho,s,\ell,\tau)$-clustered configuration.
	\end{definition}
	
	\begin{definition}
		For fixed $N \in \mathbb{N}$, $\varepsilon > 0$, $N\Delta < 1$ and $\mu \in \mathcal{R}(\Delta,\rho,s,\ell,\tau)$, let  
		\[ B^N_{\varepsilon}(\mu):=\bigg\{y\in\mathbb{C}^{2N+1}: \ \Big(\frac{1}{2N}\sum_{k=0}^{2N}\big|y_k-\hat{\mu}(k)\big|^2\Big)^{\frac{1}{2}} < \varepsilon\bigg\} \ , \]
		where $\hat{\mu}(k)$ are the Fourier coefficients as defined in \eqref{eq:noisy-data}.
	\end{definition}
	
	\begin{definition}
		Let $\mathcal{A}:=\mathcal{A}(\mathcal{R},N,\varepsilon)$ be the set of functions $\varphi$ that maps each $y\in \cup_{\mu \in \mathcal{R}}B^N_{\varepsilon}(\mu)$ to a discrete distribution $\varphi_y \in \mathcal{R}(\Delta,\rho,s,\ell,\tau)$.
	\end{definition}
	
	\begin{definition}
		For $\mu=\sum_{j=1}^{s}a_j\delta_{t_j}+b_j\delta'_{t_j}$, the norm $\|\mu\|_2$ is the discrete $\ell_2$ norm of the coefficients vector:
		\[ \|\mu\|_2 := \Big(\sum_{j=1}^{s}|a_j|^2+|b_j|^2\Big)^{\frac{1}{2}} \ . \]
	\end{definition}
	
	\begin{definition}[min-max error]\label{def:minmax}
		The $\ell^2$ min-max error for the on-the-grid model is
		\[ \mathcal{E}(\mathcal{R},N,\varepsilon) = \inf_{\varphi \in \mathcal{A}}\sup_{\mu \in \mathcal{R}}\sup_{y\in B^N_{\varepsilon}(\mu)}\|\varphi_y - \mu\|_2 \ ,\]
		where $\varphi_y := \varphi(y) \in \mathcal{R}$.
	\end{definition}
	
	\section{Main Results}\label{sec:main-results}
	
	\subsection{Optimal bounds for the smallest singular value}
	
	As in previous works on the subject, the main quantity of interest is the smallest singular value of $U_N$.
	\begin{theorem}\label{mainTheorem}
		For each $s \in \mathbb{N}$ there exists a constant $C_1=C_1(s)$ such that for any $4\tau\Delta \leq \min(\rho,\frac{1}{s^2})$, any $\boldsymbol{\xi}=(\xi_1,\dots,\xi_s)  \subset \frac{1}{s^2}\big(-\frac{\pi}{2},\frac{\pi}{2}\big]$ forming a $(\Delta,\rho,s,\ell, \tau)$-clustered configuration, and any N satisfying
		\[ \max\left(\frac{4\pi s}{\rho},4s^3\right) \leq N \leq \frac{\pi s}{\tau\Delta}\]
		we have
		\[ \sigma_{\min}(U_N(\boldsymbol{\xi})) \geq C_1\cdot{(N\Delta)}^{2\ell -1}.\]
	\end{theorem}

	\begin{theorem}
		\label{UpperTheorem}
		For each $s \in \mathbb{N}$ there exists a constant $C_2 = C_2(\ell,\tau)$ such that for any
		$\boldsymbol{\xi}=(\xi_1,\dots,\xi_s)$ forming a $(\Delta,\rho,s,\ell,\tau)$-clustered configuration, and any $N$ satisfying $N \leq \frac{1}{2\Delta}$
		we have
		\[ \sigma_{\min}(U_N(\boldsymbol{\xi})) \leq C_2 \cdot (N\Delta)^{2\ell - 1}. \]
	\end{theorem}
	
	The proofs of the above results are given in Sections \ref{sec:proof-of-main-thm} and \ref{sec:proof-of-upper-sing}, respectively. For the lower bound, we extend the decimation technique from \cite{Vandermonde} to the confluent setting.  For the upper bound we generalize the approach from \cite{Li&Liao2020} to hold \emph{for any clustered configuration} $\boldsymbol{\xi}$. Furthermore, our proof technique can be used to slightly improve the condition (2.9) in Proposition 2.10 in \cite{Li&Liao2020}
	by requiring only that $N\Delta \leq \operatorname{const}$ instead of $N^{3/2}\Delta \leq \operatorname{const}$.
	
	\subsection{Stable super-resolution of generalized spikes of order 1}
	In our setting, we assume that the spike locations are restricted to a discrete grid of step size $\Delta$. In effect, our results  show that $\mathcal{E} \asymp SRF^{4\ell-1}\varepsilon$ as $SRF\to\infty$.
	
	\begin{theorem}\label{theorem:min-max-error}
		Fix $s \geq 1, 2\leq\ell\leq s, \varepsilon > 0.$ Put $SRF := \frac{1}{N\Delta}$. Then the following hold:
		\begin{enumerate}
			\item
			For any $\rho \geq 0$, $\ell -1 \leq \tau$, and $M \geq \pi$, there exists $\Delta_0=\Delta_0(M)$ and $K$ $\geq \frac{1}{2s\pi}$ such that for every $SRF$ satisfying $K \leq SRF \leq (2sM) \cdot K$ and for all $\Delta \leq \Delta_0$, it holds that
			\[\mathcal{E}(\hat{\mathcal{R}}(\Delta,\rho,s,\ell,\tau),N,\varepsilon) \leq C_{s,\ell}{SRF}^{4\ell-1}\varepsilon \]
			for some constant $C_{s,\ell}$ depending only on s and $\ell$, where $\hat{\mathcal{R}}:=\Big\{\mu: \mu \in \mathcal{R}, \supp(\mu) \subset \frac{1}{2\pi s^2}\big(-\frac{\pi}{2},\frac{\pi}{2}\big] \Big\}$. 
			\item
			For any $\rho \geq 0$, $\ell -1 \leq \tau$ and $SRF \geq 2$, it holds that
			\[\mathcal{E}(\mathcal{R}(\Delta,\rho,s,\ell,\tau),N,\varepsilon) \geq C_{\ell,\tau}{SRF}^{4\ell-1}\varepsilon\]
			for some constant $C_{\ell,\tau}$ depending only on $\ell$ and $\tau$.
		\end{enumerate}
	\end{theorem}
	
	The proof is presented in Section \ref{sec:proof-minmax}, largely repeating the arguments from \cite{Vandermonde,Li&Liao2020}, together with  the above established bounds on $\sigma_{\min}(U_N)$.
	
	\section{Proofs}\label{sec:proofs}
	
	\subsection{Square confluent Vandermonde matrices}

	\begin{definition}
		\label{sqrconfluentVan}
		For $s \in 	\mathbb{N}$ and vector 
		$\boldsymbol{z} = (z_1, . . . , z_s)$ of pairwise distinct complex
		nodes $|{z_j}| = 1$, we define the square $2s \times 2s$ confluent Vandermonde matrix
		\[ \mathbf{U}_{2s} (\boldsymbol{z}) := [z_j^{k} \quad {kz_j^{k-1}}]^{j=1,\dots,s} _{k=0,\dots, 2s-1} \]
	\end{definition}

	\begin{theorem}[\cite{Gautschi1962}]
		\label{GautchiTheorem}
		Let $\boldsymbol{x} = (x_1, . . . , x_n)$ be a vector of pairwise distinct complex numbers and let
		
		\[ b_\lambda = \max\Big(1+|x_\lambda|,1+2\big(1+|x_\lambda|\big)\sum^n_{\nu=1 \neq \lambda}{\frac{1}{|x_\nu - x_\lambda|}}\Big) \ . \]
		Then
		
		\[ \big\|\mathbf{U}_{2n}^{-1}(\boldsymbol{x})\big\|_\infty \leq \max_{1 \leq \lambda \leq n}{\,b_\lambda \Big(\prod_{\nu=1 \neq \lambda}^n{\frac{1+|x_\nu|}{|x_\nu - x_\lambda|}}\,\Big)^2} \]
		
	\end{theorem}
	
	\begin{proposition}\label{svOfsqCv}
		Let $\boldsymbol{z} = (z_1,\dots,z_s)$ be a vector of pairwise distinct complex
		nodes with $|{z_j}| = 1, j=1,\dots,s$.
		Denote by $\delta_{j,k}$ the angular distance between $z_j$ and $z_k$:
		\[\delta_{j,k} := \delta_{j,k}(\boldsymbol{z}) = \big|\operatorname{Arg}\Big(\frac{z_j}{z_k}\big)\Big| = \Big|\operatorname{Arg}(z_j) - \operatorname{Arg}(z_k) \enspace mod (-\pi,\pi]\Big| \ . \]
		Then
		\[\sigma_{\min}\big(\mathbf{U}_{2s}(\boldsymbol{z})\big) \geq \frac{4^{2(1-s)}}{\sqrt{2s}\pi^{2(1-s)}}\min_{1 \leq j \leq s}\gamma_j\prod_{k \neq j}^{ }\delta_{j,k}^2\ ,\]
		where
		\[\gamma_j = \min\Big(\frac{1}{2}, \big(1 +\frac{8}{\pi}\sum_{k \neq j}\delta_{j,k}^{-1}\big)^{-1}\Big)\]
	\end{proposition}
	
	\begin{proof}
		By Theorem \ref{GautchiTheorem} we have
		\begin{equation}\label{eq:a}
			\big\|\mathbf{U}_{2s}^{-1}(\boldsymbol{z})\big\|_\infty \leq 2^{2(s-1)}\max_{1 \leq j \leq s}{\,b_j \prod_{k \neq j}{\big|z_k - z_j\big|}^{-2}} \ ,
		\end{equation}
		where
		\[ b_j = \max\big(2,1+4\sum^s_{k \neq j}{\big|z_j - z_k\big|^{-1}}\big) \ . \]
		For any $|\theta| \leq \frac{\pi}{2}$, we have
		\[ \frac{\pi}{2}|\theta|  \leq \sin{|\theta|}  \leq |\theta| \]
		and since for any $z_j \neq z_k$
		\[ \big|z_j - z_k\big| = \Big|1 - \frac{z_j}{z_k}\Big| = 2\sin{\Big|\frac{1}{2}\operatorname{Arg}\Big(\frac{z_j}{z_k}\Big)\Big|} = 2\sin{\Big|\frac{\delta_{j,k}}{2}\Big|}\ , \]
		we therefore obtain
		\[ \frac{\pi}{2}\delta_{j,k}  \leq \big|z_j - z_k\big|  \leq \delta_{j,k}. \]

		Plugging into \eqref{eq:a}, we have
		
		\[ \sigma_{\max}\big(\mathbf{U}_{2s}^{-1}(\boldsymbol{z})\big) \leq \sqrt{2s}\big\|\mathbf{U}_{2s}^{-1}(\boldsymbol{z})\big\|_\infty \leq \left( {\frac{4}{\pi}}\right)^{2(s-1)}\sqrt{2s}\max_{1 \leq j \leq s}{\,b_j \prod_{k \neq j}^s{\delta_{j,k}^{-2}}} \]
		and
		\[ b_j = \max\big(2, 1+ \frac{8}{\pi}\sum^s_{k \neq j}{\delta_{j,k}^{-1}}\big). \]
		This finishes the proof with $\gamma_j := b_j^{-1}$.
	\end{proof}

	\subsection{Proof of Theorem \ref{mainTheorem}. }\label{sec:proof-of-main-thm}
	\subsubsection{Overview of the proof}
	First we use the \textbf{Decimation} technique that has first been introduced in \cite{Vandermonde}. It states that there exists a certain blow-up factor $\lambda$ such that the mapped nodes $\{e^{i\lambda x_j}\}$ attain "good" separation properties.
	Second, for any such $\lambda$ of order $O(N)$, we can partition the rectangular confluent Vandermonde matrix into squared well-conditioned confluent matrices and use this partition to bound $\sigma_{\min}$ from below.
	
	In order to use the corresponding results from \cite{Vandermonde}, we introduce an auxiliary bandwidth parameter $\Omega$.
	
	\begin{definition}
		\label{bandLimReconfVan}
		For $N, s \in 	\mathbb{N}$, a vector 
		$\boldsymbol{x} = (x_1, . . . , x_s)$ of pairwise distinct real
		nodes $x_j \in \big( - \frac{\pi}{2} , \frac{\pi}{2} \big]$, and a bandwidth parameter $0 < \Omega \leq 2N$, let $\boldsymbol{\xi} = (\xi_1, \dots , \xi_s)$ where $\xi_j = \frac{x_j \Omega}{N}$. Then we define
		\[ U_N (\boldsymbol{x}, \Omega) := U_N (\boldsymbol{\xi}) = U_N \big(\frac{\Omega}{N}\boldsymbol{x}\big) = \frac{1}{\sqrt{2N}}\Big[\exp\big(ik\frac{x_j\Omega}{N}\big) \quad k\exp\big(i(k-1)\frac{x_j\Omega}{N}\big)\Big]^{j=1,\dots,s} _{k=0,\dots, 2N} \in \mathbb{C}^{(2N+1)\times (2s)}.\]
		
	\end{definition}
	\subsubsection{The existence of an admissible decimation}
	
	We can now use a key result from \cite{Vandermonde}.
	\begin{lemma}[Lemma 4.1 in \cite{Vandermonde}]\label{decemation}
		Let $\boldsymbol{x}$ form a $(\Delta, \rho, s, \ell, \tau)$ clustered configuration, and suppose that $\frac{4\pi s}{\rho}\leq\Omega\leq\frac{\pi s}{\tau\Delta}$. Then, for any $0 \leq \xi \leq 1$, there exists a set $I \subset \Big[\frac{\Omega}{2s}, \frac{\Omega}{s}\Big]$ of total measure $\frac{\Omega}{2s}
		\xi$ such that for every $\lambda \in I$, the following holds for every $x_j \in \boldsymbol{x}$:
		\begin{enumerate}
			\item
			\[ \|\lambda y - \lambda x_j\|_{\tilde{\mathbb{T}}} \geq \lambda\Delta \geq \frac{\Delta\Omega}{2s} \quad \quad \forall y \in \boldsymbol{x}^{(j)} \setminus \{x_j\}\]
			\item
			\[ \|\lambda y - \lambda x_j\|_{\tilde{\mathbb{T}}} \geq \frac{1-\xi}{s^2}\pi \quad \quad \forall y \in \boldsymbol{x}\setminus \boldsymbol{x}^{(j)}\]
		\end{enumerate} 
		Furthermore, the set $I^c := \Big[\frac{\Omega}{2s}, \frac{\Omega}{s}\Big] \setminus I$ is a union of at most $\frac{s^2}{2}\Big\lceil\frac{\Omega}{4s}\Big\rceil$ intervals.
	\end{lemma}
	Fix $\xi = \frac{1}{2}$ and consider the set $I$ given by the above Lemma. Let us also fix a finite and positive integer $N$ and consider the set of $2N + 1$ equispaced points in $[0, 2\Omega]$:
	\[
	P_N := \big\{k\frac{\Omega}{N}\big\}_{k=0,\dots,2N}.
	\]
	\begin{proposition}\label{pn}
		If $N > 2s^3\Big\lceil\frac{\Omega}{4s}\Big\rceil$, then $P_N \cap I \neq \emptyset.$
	\end{proposition}
	\begin{proof}
		Exactly as the proof of Proposition 4.2 in \cite{Vandermonde}.
	\end{proof}
	
	We are now in a position to extend the main result from \cite{Vandermonde} to the confluent setting.
	\begin{theorem}\label{preMainTheorem}
		There exists a constant $C = C (s)$ such that for any
		$\boldsymbol{x}$ forming a $(\Delta , \rho , s, \ell , \tau )$-clustered configuration, and any $\Omega$ satisfying 
		\[ \frac{4\pi s}{\rho} \leq \Omega \leq \frac{\pi s}{\tau \Delta}, \]
		we have
		\[ \sigma_{\min}\big(U_N (\boldsymbol{x},\Omega)\big) \geq C \cdot (\Delta \Omega )^{2\ell - 1} \quad \quad \textrm{whenever} \quad N > 2s^3\biggl\lceil \frac{\Omega}{4s}\biggr\rceil \ . \]
	\end{theorem}
	\begin{proof}
		Similarly to the proof of theorem 3.2 in \cite{Vandermonde}, for any subset $R \subset \{ 0, \dots, 2N\}$ let $U_{N,R}$ be the submatrix of $U_N$ containing only the rows in $R$. In particular, 
		if $\{ 0, \dots, 2N\}=R_1 \dot\cup \dots \dot\cup R_p$ then
		$$
		\sigma^2_{\min}(U_N) \geq \sum^p_{n=1}{\sigma^2_{\min}(U_{N,R_n})}.
		$$
		By Lemma \ref{decemation} and Proposition \ref{pn}, there exists $m \in \mathbb{N}, \ 0 \leq m \leq 2N$ such that
		\[ u_j := x_j\frac{\Omega}{N}m = \lambda x_j \]
		with
		\begin{eqnarray}\label{ujk-bounds}
			\begin{split}
			\frac{\tau}{2s}(\Delta \Omega ) &\geq \|u_j - u_k\|_{\tilde{\mathbb{T}}} \geq \frac{1}{2s}(\Delta \Omega ) &\forall x_k &\in \boldsymbol{x}^{(j)}\setminus \{ x_j\};\\
			\pi &\geq \|u_j - u_k\|_{\tilde{\mathbb{T}}} \geq \frac{\pi}{2s^2} &\forall x_k &\in \boldsymbol{x}\setminus \boldsymbol{x}^{(j)}.
			\end{split}
		\end{eqnarray}
		Since $\lambda \leq \frac{\Omega}{s}$ we conclude that $2ms \leq 2N$. 
		
		We will divide $U_N$ to $m$ squared matrices of size $2s \times 2s$ in the following form:
		\begin{align*}
			R_0 &= \{0, m, \dots, (2s-1)m \}, \\
			R_1 &= \{1, m+1, \dots, (2s-1)m+1 \},  \\
			& \vdots \\
			R_{m-1} &= \{m-1, 2m-1, \dots, 2sm-1 \}.
		\end{align*}
		For $k=0,1,\dots,m-1$ each $U_{N,R_k}$ is a square confluent Vandermonde matrix, and it can be checked by direct computation that 
		\[ U_{N,R_k}(\boldsymbol{\nu}) = \frac{1}{\sqrt{2N}}\mathbf{U}_{2s}(\boldsymbol{\nu})D(\boldsymbol{z},m)T(\boldsymbol{z},k)
		\]
		where $\boldsymbol{\nu} = \big\{\exp{(iu_j)}\big\}^s_{j=1}$ and $\boldsymbol{z} = \big\{\exp{(ix_j\frac{\Omega}{N})}\big\}^s_{j=1}$, with
		\[ D(\boldsymbol{z},m) = \diag{1,\dots,1,mz_{1}^{m-1},\dots,mz_{s}^{m-1}} \ , \]
		\[  T(\boldsymbol{z},r) := \begin{bmatrix}
			z_1^r  &  \dots &   0    & {r}z_1^{r-1} & \dots    & 0 \\
			\vdots & \ddots & \vdots &              &  \ddots  & \vdots \\
			0    & \dots  & z_s^r  &       0      &  \dots   & {r}z_s^{r-1} \\
			0    &  \dots &   0    &     z_1^r    & \dots    & 0 \\
			\vdots & \ddots & \vdots &              &  \ddots  & \vdots \\
			0    & \dots  &    0   &       0      &  \dots   & z_s^r 
		\end{bmatrix} \ . \]
		
		Recall the well-known formula for a block matrix inverse.
		\begin{lemma}[e.g. \cite{MatrixAnalysis}]
			\label{invOfUT}
			Consider the block upper triangular matrix
			\[  \begin{bmatrix}
				A & B \\
				0 & D
			\end{bmatrix}.
			\]
			It is invertible if and only if both A and D are invertible, and its inverse is given by
			\[  \begin{bmatrix}
				A^{-1} & -A^{-1}BD^{-1} \\
				0 & D^{-1}
			\end{bmatrix}
			\]
		\end{lemma}
		\begin{lemma}
			\label{lemma1}
			For $r \in \mathbb{Z}$, $s, m \in \mathbb{N}$, $m \neq 0$ and vector 
			$\boldsymbol{z} = (z_1, . . . , z_s)$ of pairwise distinct complex
			nodes with $|{z_j}| = 1$ we have
			\[ \big\|P^{-1}(\boldsymbol{z},r,m)\big\|_\infty = \bigg|\frac{r}{m}\bigg|+1\ , \]
			where
			\[ P(\boldsymbol{z},r,m) = D(\boldsymbol{z},m)T(\boldsymbol{z},r) \ . \]
		\end{lemma}
		
		\begin{proof}
			By direct computation, 
			\[  P(\boldsymbol{z},r,m) := \begin{bmatrix}
				z_1^r  &  \dots &   0    & {r}z_1^{r-1} & \dots    & 0 \\
				\vdots & \ddots & \vdots &              &  \ddots  & \vdots \\
				0    & \dots  & z_s^r  &       0      &  \dots   & {r}z_s^{r-1} \\
				0    &  \dots &   0    &     {m}z_1^{r+m-1}    & \dots    & 0 \\
				\vdots & \ddots & \vdots &              &  \ddots  & \vdots \\
				0    & \dots  &    0   &       0      &  \dots   & {m}z_s^{r+m-1} 
			\end{bmatrix}= \begin{bmatrix}
				A & B \\
				0 & C
			\end{bmatrix}\]
			
			where
			\[ A := \diag{z_1^r, \dots, z_s^r}, \quad \quad  B := \diag{{r}z_1^{r-1}, \dots, {r}z_s^{r-1}},
			\]
			\[ C := \diag{{m}z_1^{m+r-1}, \dots, {m}z_s^{m+r-1}}. \]
			By Lemma \ref{invOfUT} we get
			\[  P^{-1}(\boldsymbol{z},r,m) = \begin{bmatrix}
				A^{-1} & -A^{-1}BC^{-1} \\
				0 & C^{-1}
			\end{bmatrix} \ ,
			\]
			where $-A^{-1}BC^{-1} = \diag {{-\frac{r}{m}}{z_1^{-(r+m)}}, \dots, {-\frac{r}{m}}{z_s^{-(r+m)}}}$.
			Thus
			\[ \big\|P^{-1}(\boldsymbol{z},r,m)\big\|_\infty = \max_{k}  \big|z_k^{-r}\big| + \bigg|{-\frac{r}{m}}{z_k^{-(r+m)}}\bigg| = 1+\bigg|\frac{r}{m}\bigg| \ .
			\]
			
		\end{proof}
		
		Now, let us take a look at $\gamma_j$ from Proposition \ref{svOfsqCv}:
		\[ \gamma_j = \min\Big(\frac{1}{2},\big(1+\frac{8}{\pi}\big(\sum^s_{k \neq j}{\delta_{j,k}^{-1}}\big)\big)^{-1}\Big) \ , \]
		where
		\[\delta_{j,k}:=\delta_{j,k}(\boldsymbol{\nu}) \ . \]
		We will show two properties:
		\begin{enumerate}
			\item
			\[  \sum_{k \neq j}{\delta_{j,k}^{-1}} \leq \frac{2s\ell}{\Delta \Omega} + \frac{2(s-\ell)s^2}{\pi} \leq \frac{2s\ell \pi + 2(s-\ell)s^2 \Delta \Omega}{\Delta \Omega \pi}\]
            
			\begin{align*}
				(1+\frac{8}{\pi}\sum_{k \neq j}{\delta_{j,k}^{-1}})^{-1} \geq \frac{\Delta \Omega \pi^2}{\Delta\Omega\pi^2 + 16s(\ell \pi + (s-\ell)s \Delta \Omega)}
			\end{align*}
			Given that $\Delta \Omega < \frac{\pi s}{\tau} < \pi s$, we have
			\[(1+\frac{8}{\pi}\sum_{k \neq j}{\delta_{j,k}^{-1}})^{-1} \geq \Delta\Omega\frac{\pi^2}{\pi^3 s+16s(\ell\pi+(s-\ell)s\pi s)} \geq \Delta\Omega\frac{\pi}{\pi^2s+16s(\ell +(s-\ell)s^2)} \geq \Delta\Omega\frac{\pi}{\pi^2 s+16s^2 +16s^4}\]
			\begin{equation}\label{bj:p1}\tag{P1}
				\Rightarrow \big(1+\frac{8}{\pi}\sum_{k \neq j}{\delta_{j,k}^{-1}}\big)^{-1} \geq  \kappa(s)\Delta \Omega \ ,
			\end{equation}
			where
			\[ \kappa(s) = \frac{\pi}{\pi^2 s+16s^2 +16s^4} \ . \]
			\\ \\
			\item
			Using $ 2 \leq \ell \leq s$ and $\Delta \Omega \tau < \pi s$ we get
                 
			\[ \big(1+\frac{8}{\pi}\sum_{k \neq j}{\delta_{j,k}^{-1}}\big)^{-1} \leq  \bigg(1+\frac{8}{\pi}
   \bigg\{\big(\frac{2s}{\tau\Delta\Omega}\big)(\ell -1) + \frac{1}{\pi}(s-\ell)\bigg\}\bigg)^{-1}\]
			
			\[1+\frac{8}{\pi}\sum_{k \neq j}{\delta_{j,k}^{-1}} \geq 1+\frac{8}{\pi}\big(\frac{2s\pi(\ell -1)+(s-\ell)\tau\Delta\Omega}{\tau\Delta\Omega\pi}\big) = 1+ \frac{16s\pi(\ell -1)+8(s-\ell)\tau\Delta\Omega}{\tau\Delta\Omega\pi^2} \]
			
			\[ = \frac{\tau\Delta\Omega\pi^2 + 16s\pi(\ell -1)+8(s-\ell)\tau\Delta\Omega}{\tau\Delta\Omega\pi^2} \geq \frac{\tau\Delta\Omega(\pi^2 + 16(\ell -1)+8(s-\ell))}{\tau\Delta\Omega\pi^2} \geq 2\]
			
			\[\Rightarrow \big(1+\frac{8}{\pi}\sum_{k \neq j}{\delta_{j,k}^{-1}}\big)^{-1} \leq  \frac{1}{2} . \]
			
			\begin{equation}\label{bj:p2}\tag{P2}
				\Rightarrow \gamma_j = \min\Big(\frac{1}{2},\big(1+\frac{8}{\pi}\sum_{k \neq j}{\delta_{j,k}^{-1}}\big)^{-1}\Big) = \big(1+\frac{8}{\pi}\sum_{k \neq j}{\delta_{j,k}^{-1}}\big)^{-1}.
			\end{equation}
   
		\end{enumerate}
	
		Using Proposition \ref{svOfsqCv} and Lemma \ref{lemma1} we are going to bound from below the smallest singular value of the \textit{square} confluent Vandermonde matrix:
		\begin{align*}
			\sigma_{\min}(U_{N,R_r}) &= \sigma_{\min}\big(\frac{1}{\sqrt{2N}}\mathbf{U}_{2s}(\boldsymbol{\nu})D(\boldsymbol{z},m)T(\boldsymbol{z},r)\big) \geq \sigma_{\min}\big(\mathbf{U}_{2s}(\boldsymbol{\nu})\big)\big(\sqrt{2N}\sqrt{2s}\big\|P^{-1}(\boldsymbol{z},r,m)\big\|_\infty\big)^{-1} \\
			&\geq \frac{1}{2s\sqrt{2N}}{\bigg(\frac{4}{\pi}}\bigg)^{2(1-s)}{\bigg(1+\bigg|\frac{r}{m}\bigg|\bigg)}^{-1}\min_{1 \leq j \leq s}{\,\gamma_j \prod_{k \neq j}^s{\delta_{j,k}^{2}(\boldsymbol{\nu})}} \\
%			&{\color{red}\geq 
			%\frac{{4}^{2(1-s)}}{2s{\pi}^{2(1-s)}}}\frac{\kappa(s)}{\sqrt{2N}}{\bigg(1+\bigg|\frac{r}{m}\bigg|\bigg)}^{-1}{\color{red}\frac{(\Delta \Omega)^{2\ell-1}}{(2s)^{2\ell-2}}\big(\frac{\pi}{2s^2}\big)^{2(s-\ell)}} \\
			&\geq \frac{\tilde{\kappa}(s)}{\sqrt{2N}}{\bigg(1+\bigg|\frac{r}{m}\bigg|\bigg)}^{-1}(\Delta \Omega)^{2\ell-1}
		\end{align*}
		for some constant $\tilde{\kappa}(s)$.
		Ahead of the last step we used \eqref{ujk-bounds}, properties \eqref{bj:p1}, \eqref{bj:p2} and the fact that $1 < \ell \leq s$.
		
		Finally, we can bound from below the smallest singular value of the \textit{rectangular} confluent Vandermonde matrix:
		\begin{align*}
			\sigma^2_{\min}(U_N) &\geq \sum_{r=0}^{m-1}{\bigg(1+\bigg|\frac{r}{m}\bigg|\bigg)^{-2}\frac{\tilde{\kappa}^2(s)}{2N}(\Delta \Omega)^{2(2\ell-1)}} \\
			&\geq \frac{\tilde{\kappa}^2(s)}{2N}(\Delta \Omega)^{2(2\ell-1)}\sum_{r=0}^{m-1}{(2)^{-2}} \\
			&\geq \frac{m\tilde{\kappa}^2(s)}{8N}(\Delta \Omega)^{4\ell-2} \\
			&\geq \frac{\tilde{\kappa}^2(s)}{16s}(\Delta \Omega)^{4\ell-2}.
		\end{align*}
		We used the fact that $m = \frac{\lambda N}{\Omega} \geq \frac{\Omega N}{2s\Omega} = \frac{N}{2s}$.
		\\
		
		To summarize, the final result for Theorem \ref{preMainTheorem} is
		\[ \sigma_{\min}(U_N(\boldsymbol{x},\Omega)) \geq C_1(s)(\Delta \Omega)^{2\ell-1} \ , \quad C_1(s):=\frac{\tilde{\kappa}(s)}{\sqrt{16s}}. \qedhere \]
	\end{proof}
	
	\begin{proof}[Proof of Theorem \ref{mainTheorem}]
		Similar to  \cite[Corollary 3.6]{Vandermonde}), for $\Omega:=\frac{N}{s^2}$ and any $\boldsymbol{\xi} \subset \frac{1}{s^2}\big(-\frac{\pi}{2} , \frac{\pi}{2}\big]$ forming a $(\Delta,\rho,s,\ell,\tau)$-clustered configuration with the conditions of Theorem \ref{mainTheorem}, we have $\boldsymbol{x}=\frac{N}{\Omega}\boldsymbol{\xi} \subset \big(-\frac{\pi}{2} , \frac{\pi}{2}\big] $ which forms a $(\tilde{\Delta},\tilde{\rho},s,\ell,\tau)$-clustered configuration with $\tilde{\Delta}:=s^2\Delta$ and $\tilde{\rho}:=s^2\rho$. \\
		Clearly, $4\tau \tilde{\Delta} \leq s^2\rho=\tilde{\rho}$ and also
		\[ \Omega s^2 = N \geq 4s^3 \ \Rightarrow \  \frac{\Omega}{4s} \geq 1 \ \Rightarrow \ \frac{2\Omega}{4s} \geq \bigg\lceil \frac{\Omega}{4s} \bigg\rceil \ \Rightarrow \ N=\Omega s^2 \geq 2s^3 \bigg\lceil \frac{\Omega}{4s} \bigg\rceil \ , \]
		thus the conditions of Theorem \ref{preMainTheorem} are satisfied for $\boldsymbol{x},\Omega,\tilde{\rho},\tilde{\Delta},\tau$.
		Therefore
		\[ \sigma_{\min}\big(U_N (\boldsymbol{\xi})\big) = \sigma_{\min}\big(U_N(\boldsymbol{x},\Omega)\big) \geq C \cdot (\tilde{\Delta} \Omega)^{2\ell - 1} = C \cdot \bigg(\frac{N}{\Omega}\Delta\Omega\bigg)^{2\ell - 1} = C \cdot (\Delta N)^{2\ell - 1} \ , \]
		finishing the proof of Theorem \ref{mainTheorem}.
	\end{proof}
	
	\subsection{Proof of Theorem \ref{UpperTheorem}. }\label{sec:proof-of-upper-sing}
	\begin{definition}\label{unnormlizedConf}
		For $M, s \in 	\mathbb{N}$ and a vector 
		$\boldsymbol{\omega} = (\omega_1, . . . , \omega_s)$ of pairwise distinct real
		nodes $\omega_j \in \mathbb{T}$, let $\Phi_M$ denote the $(M+1)\times 2s$ confluent Vandermonde matrix
		\[\Phi_M(\boldsymbol{\omega}) = \begin{pmatrix}
			1 & \dots & 1 & 0 & \dots & 0
			\\
			z_1 & \dots & z_s & 1 & \dots & 1
			\\
			z_1^2 & \dots & z_s^2 & 2z_1 & \dots & 2z_s
			\\
			\vdots & \dots & \vdots & \vdots & \dots & \vdots
			\\
			z_1^M & \dots & z_s^M & Mz_1^{M-1} & \dots & Mz_s^{M-1}
			
		\end{pmatrix} \ , \]
		and let $V_M$ denote the $(M+1)\times 2s$ pascal Vandermonde matrix
		\[V_M(\boldsymbol{\omega}) = \begin{pmatrix}
			1 & \dots & 1 & 0 & \dots & 0
			\\
			z_1 & \dots & z_s & 2\pi iz_1 & \dots & 2\pi i z_s
			\\
			z_1^2 & \dots & z_s^2 & 4\pi iz_1^2 & \dots & 4\pi i z_s^2
			\\
			\vdots & \dots & \vdots & \vdots & \dots & \vdots
			\\
			z_1^M & \dots & z_s^M & M2\pi iz_1^M & \dots & 2M\pi i z_s^M
			
		\end{pmatrix} , \]
		where $z_j = \exp(-2\pi i\omega_j)$ and $\mathbb{T}$ is the periodic interval $[0,1)$.
	\end{definition}

	By direct computation we get
	\[ V_M = \Phi_M H\]
	with $H = \diag{1,\dots,1,2\pi iz_1,\dots,2\pi iz_s}$.

	Inspired by the proof of Proposition 2.10 in \cite{Li&Liao2020}, we will consider $\boldsymbol{\omega}=\frac{\boldsymbol{\xi}}{-2\pi}+\frac{1}{2}$ where $\boldsymbol{\xi}$ is a $(\Delta,\rho,s,\ell,\tau)$- clustered configuration and a suitable vector $u$ in order to obtain an upper bound for
	\[ \sigma_{\min}(\Phi_M(\boldsymbol{\omega})) = \min_{u \in \mathbb{C}^{2s}, u \neq 0} \frac{\|\Phi_M(\boldsymbol{\omega}) u\|_2}{\|u\|_2}.\]
	Put $\alpha:=M\Delta$, assume that $M < \frac{1}{\Delta}$ and let $\boldsymbol{\omega} = \{\omega_j\}^s_{j=1}$ be defined w.l.o.g by $\omega_j = \tau_j\frac{\alpha}{M}$ where $\tau_1=0$, $\tau_j < \tau_{j+1}$, $\tau_{\ell} \leq \tau$ for $1 \leq j \leq \ell$, while $\{\omega_j\}_{\ell +1}^{s}$ are arbitrary.

	\begin{definition}\label{def:u-vector-def}
		We consider the vector $u \in \mathbb{C}^{2s}$ defined by
		\begin{align*}\label{def:u} \tag{$\star$}
			u_j &:= {\bigg(\frac{\alpha}{M}\bigg)}^{2\ell -1}A_j  &1\leq j \leq \ell 
			\\  u_{s+j} &:= {\bigg(\frac{\alpha}{M}\bigg)}^{2\ell -1}2\pi iz_jB_j &1\leq j \leq \ell 
			\\
			u_j &= u_{s+j} = 0 &\text{otherwise} \ ,
		\end{align*}
		where $z_j := \exp(-2\pi i\omega_j)$ and $A_j, B_j$ are as given by equation \eqref{eq:AB} from appendix \ref{appendix:taylor}.
	\end{definition}

	Let	$\tilde{u}_j := u_j, \ \tilde{u}_{s+j} := \frac{z_j^{-1}}{2\pi i}u_{s+j}$ for $1\leq j \leq \ell$ and $\tilde{u}_j = \tilde{u}_{s+j} = 0$ otherwise.
	To estimate $\|\Phi u\|_2$, we identify $u$ with the discrete distribution
	\begin{equation}\label{eq:I}
		\mu := \sum_{j=1}^{\ell}\tilde{u}_j\delta_{\tau_j\frac{\alpha}{M}} + \tilde{u}_{s+j}\delta'_{\tau_j\frac{\alpha}{M}}.
	\end{equation}  
	\\
	We also define a modified Dirichlet kernel $D_M \in {C}^{\infty}(\mathbb{T})$ by 
	\begin{equation}\label{eq:II}
		D_M(\omega) := \sum_{m=0}^{M}\exp(2\pi im\omega) \ .
	\end{equation} 
	\\
	\begin{lemma}\label{lemma:L2}
		For $\mu$ and $D_M$ as defined in \eqref{eq:I}, \eqref{eq:II}, the following is true:
		\[\sum_{m=0}^{M}|\hat{\mu}(m)|^2 = \|\mu * D_M\|^2_{L^2(\mathbb{T})}.\]  
	\end{lemma}
	The proof of the above lemma is in appendix \ref{appendix:L2}.
	Thus, observe the following:
	\begin{align*} 
		\|\Phi_M u\|_2^2 = \|V_MH^{-1}u\|_2^2 = \|V_M\tilde{u}\|_2^2 = \sum_{m=0}^{M}|{(V_M\tilde{u})}_m|^2 = \sum_{m=0}^{M}|\hat{\mu}(m)|^2 = \|\mu * D_M\|^2_{L^2(\mathbb{T})}.
	\end{align*}
	\\
	As shown in appendix \ref{appendix:taylor}, we see that for all $\omega \in \mathbb{T}$
	\begin{equation}\label{eq:muCovD}
		\begin{split}
			(\mu * D_M)(\omega) &= \sum_{j=1}^{\ell}\Bigg(\tilde{u}_{j}D_M\bigg(\omega-\frac{\tau_j\alpha}{M}\bigg) + \tilde{u}_{s+j}D'_M\bigg(\omega-\frac{\tau_j\alpha}{M}\bigg)\Bigg) \\
			&= \bigg(\frac{\alpha}{M}\bigg)^{2\ell -1}D_M^{(2\ell -1)}(\omega) + \bigg(\frac{\alpha}{M}\bigg)^{2\ell -1}\{R_A(\omega)+R_B(\omega)\} \ ,
		\end{split}
	\end{equation}
	where $R_A(\omega)$ and $R_B(\omega)$ are written explicitly in appendix \ref{appendix:taylor}.
	
	By the Bernstien inequality for trigonometric polynomials \cite{Bernstein1912}, we have
	\begin{equation}\label{eq:bernD}
		\|D_M^{(2\ell -1)}\|_{L^2(\mathbb{T})} \leq {(2\pi M)}^{2\ell -1}\|D_M\|_{L^2(\mathbb{T})}
		= \sqrt{(M+1)}(2\pi M)^{2\ell -1} \ .
	\end{equation}

	\begin{lemma}\label{lemma:RaRb}
		For $R_A(\omega)$ and $R_B(\omega)$ as defined in appendix \ref{appendix:taylor} in the appendix, we have
		\begin{align*} \|R_A(\omega) \|_{L^2(\mathbb{T})} &\leq \bigg(\frac{\alpha}{M}\bigg)^{2\ell}\sqrt{M+1}(2\pi M)^{2\ell}\frac{\tau^{2\ell}}{(2\ell -1)!}\sum_{i=1}^{\ell}\big|A_i\big|, \\
			\|R_B(\omega) \|_{L^2(\mathbb{T})} &\leq \bigg(\frac{\alpha}{M}\bigg)^{2\ell -1}\sqrt{M+1}(2\pi M)^{2\ell}\frac{\tau^{2\ell -1}}{(2\ell -1)!}\sum_{i=1}^{\ell}\big|B_i\big|. 
		\end{align*}
	\end{lemma}
	
	\begin{lemma}\label{lemma:AB}
		For $1 \leq i \leq \ell$ and $A_j$, $B_j$ as defined in appendix \ref{appendix:taylor}, we can bound the following expressions as follows:
		\begin{align*} \sum_{i=1}^{\ell}\big|A_i\big| &\leq C_A(\ell,\tau){\bigg(\frac{M}{\alpha}\bigg)^{2\ell -1}},\\
			\sum_{i=1}^{\ell}\big|B_i\big| &\leq C_B(\ell,\tau){\bigg(\frac{M}{\alpha}\bigg)^{2\ell -2}}
			\ .
		\end{align*}
		\begin{comment}
		where
		\[ C_A(\ell) = 2{\ell}^2(2\ell -1)!{\bigg(\frac{\ell +1}{\ell}\bigg)}^{2\ell -1} \ , \quad C_B(\ell) = 2{\ell}(2\ell -1)!{\bigg(\frac{\ell +1}{\ell}\bigg)}^{2\ell -2} \ .
		\]
		\end{comment}
	\end{lemma}
	The proofs of Lemmas \ref{lemma:RaRb} and \ref{lemma:AB} are shown in appendix \ref{appendix:RaRb} and \ref{appendix:AB} respectively.
	
	Combining \eqref{eq:bernD} and Lemmas \ref{lemma:AB} and \ref{lemma:RaRb} we get:
	\begin{equation}\label{eq:1}
		\begin{split}
			\|\mu * D_M\|_{L^2(\mathbb{T})} &\leq \bigg(\frac{\alpha}{M}\bigg)^{2\ell -1}\bigg(\|D_M^{(2\ell -1)}\|_{L^2(\mathbb{T})} + \|R_A(\omega) \|_{L^2(\mathbb{T})} + \|R_B(\omega) \|_{L^2(\mathbb{T})} \bigg) \\ 
			&\leq \sqrt{M+1}(2\pi \alpha)^{2\ell -1}(1+\tilde{C}(\ell,\tau)2 \pi \alpha) \ .
		\end{split}
	\end{equation}
	The proof of the following lemma is in appendix \ref{appendix:u}.
	\begin{lemma}\label{lemma:u}
		Let $u \in \mathbb{C}^{2s}$ be defined by equation \eqref{def:u} from Definition \ref{def:u-vector-def}, then
		\begin{equation}\label{eq:2}
			\|u\|_2 \geq \tilde{C}_3(\ell,\tau):= \frac{(2\ell -1)!}{4\ell^3\tau^{2\ell -1}}.
		\end{equation}
	\end{lemma}
	Combining \eqref{eq:1} and \eqref{eq:2} we get:
	\[ \frac{\|\Phi_M u\|_2}{\|u\|_2} \leq \frac{\sqrt{M+1}(2\pi \alpha)^{2\ell -1}(1+\tilde{C}(\ell,\tau)2 \pi \alpha)}{\tilde{C}_3(\ell,\tau)} \leq \hat{C}(\ell,\tau)\sqrt{M+1}(2\pi \alpha)^{2\ell -1} \ . \]
	\begin{proposition}\label{unitaryEqui}
		For $N, s, d \in 	\mathbb{N}$ and vector 
		$\boldsymbol{\xi} = (\xi_1, . . . , \xi_s)$ of pairwise distinct real
		nodes $\xi_j \in ( -\pi , \pi ]$, let $\Phi_{2N}$ be as in definition \ref{unnormlizedConf}. Then, the following decomposition holds:
		\[ \tilde{\Phi}_{2N}(\boldsymbol{\eta}) := \frac{1}{\sqrt{2N}}\Phi_{2N}\bigg(\frac{\boldsymbol{\xi}}{-2\pi}+\frac{1}{2}\bigg) = \frac{1}{\sqrt{2N}}E_1\Phi_{2N}\bigg(\frac{\boldsymbol{\xi}}{-2\pi}\bigg)E_2 = E_1U_{N}(\boldsymbol{\xi})E_2 \ , \]
		where
		\[ E_1 = \diag{1,e^{-2\pi i\frac{1}{2}}, \dots,e^{-2\pi iM\frac{1}{2}}}_{2s\times 2s} \quad and \quad E_2 = \diag{1,\dots,1,e^{2\pi i\frac{1}{2}},\dots,e^{2\pi i\frac{1}{2}}}_{2s\times 2s} \ . \]
		Therefore, $\tilde{\Phi}_{2N}(\boldsymbol{\eta})$ and $U_{N}(\boldsymbol{\xi})$ are unitary equivalent and thus have the same singular values.
	\end{proposition}
	Finally, by Proposition \ref{unitaryEqui} and setting $M=2N$ we get:
	\[\frac{\boldsymbol{\xi}}{-2\pi}+\frac{1}{2} = \boldsymbol{\omega} \Rightarrow \boldsymbol{\xi}=2\pi \boldsymbol{\omega}-\pi \ , \] 
	and
	\[ \sigma_{\min}(U_{N}(\boldsymbol{\xi})) = \sigma_{\min}\Bigg(\tilde{\Phi}_{2N}\bigg(\frac{\boldsymbol{\xi}}{-2\pi}+\frac{1}{2}\bigg)\Bigg) = \sigma_{\min}(\tilde{\Phi}_{2N}(\boldsymbol{\omega}))) \leq C_2(\ell,\tau)(2\pi \alpha)^{2\ell -1} \ , \]
	completing the proof of Theorem \ref{UpperTheorem}.
	
	\subsection{Proof of Theorem \ref{theorem:min-max-error}}\label{sec:proof-minmax}
	\subsubsection{Notation}
	\begin{definition}[Pascal-Vandermonde matrix]
		For $\mathbf{t}=(t_1,\dots,t_{s})\in\mathbb{R}^{1\times s}$ and $z_j = e^{it_j}$ let
		\[ H:=H(\mathbf{t}) = \diag{1,\dots,1,-iz_1,\dots,-iz_s)_{2s\times 2s}, \quad P_{N}(\mathbf{t})=U_{N}(\mathbf{t})H(\mathbf{t}}.\]
	\end{definition}
	
	Every discrete distribution $\mu=\sum_{j=1}^{s}a_j\delta_{t_j}+b_j\delta'_{t_j} \in \mathcal{R}$ can be identified with a sparse vector $x_{\mu} \in \mathbb{C}^{2G}_{\mathcal{R}}\subset \mathbb{C}^{2G}$, where $G=G(\Delta)=2M+1$ and $M= \big\lfloor \frac{\pi}{2\Delta} \big\rfloor$ from definition \ref{grid}, $\|x\|_0=\#(i|x_i \neq 0)$,
	\begin{equation}\label{xMu}
		(x_{\mu})_i := \begin{cases} 
			a_j & t_j=(-M+i-1)\Delta \in \mathbf{t} \ \land \ 1 \leq i \leq G\\
			-ib_{j} & t_j=(-M+i-1-G)\Delta \in \mathbf{t} \ \land \ G+1 \leq i \leq 2G \\
			0 & \text{otherwise} 
		\end{cases}
	\end{equation}
	for $j = 1,\dots, 2s$ and
	\[\mathbb{C}^{2G}_{\mathcal{R}(\Delta,\rho,s,\ell,\tau)}:=\bigg\{x_{\nu}: \nu \in \mathcal{R}(\Delta,\rho,s,\ell,\tau), x_{\nu} \in \mathbb{C}^{2G}, \|x_{\nu}\|_0 \leq 2s \ \text{and $x_{\nu}$ as defined in $\eqref{xMu}$} \bigg\} \ .\]

	A direct computation shows that for every $\omega \in \mathbb{R}$
	\[ \hat{\mu}(\omega) = \sum_{j=1}^{s}a_je^{i\omega t_j}- i\omega b_je^{i\omega t_j}.\]
	Thus we can write
	\[ \begin{pmatrix}
		\hat{\mu}(-M) \\
		\hat{\mu}(-M+1) \\
		\vdots \\
		\hat{\mu}(M)
	\end{pmatrix}_{(G \times 1)} = \begin{pmatrix}
		\mathcal{F}_G & \mathcal{F}'_G
	\end{pmatrix} x_{\mu},
	\]
	where $\mathcal{F}_G = \big[e^{ikj\Delta}\big]_{k=-M,...,M}^{j=-M,...,M}$ is a $G\times G$ matrix and $\mathcal{F}'_G =\diag{-M,...,M}\mathcal{F}_G$.

	\begin{corollary}
		Assume that $2N \leq M$, let $F_{2N+1}$ and $F'_{2N+1}$ be the $2N+1$ rows $\{M,\dots,M+2N\}$ of $\mathcal{F}_G$ and $\mathcal{F}'_G$ respectively. In addition, let $\tilde{\mathcal{F}}_{2N} = \begin{pmatrix}
			F_{2N+1} & F'_{2N+1}
		\end{pmatrix}$. Then,
		\[ \tilde{\mathcal{F}}_{2N}x_{\mu} = \sqrt{2N}P_{N}v_{\mu} = \sqrt{2N}U_{N}w_{\mu}=(\hat{\mu}(0),\dots,\hat{\mu}(2N))^T \ , \] 
		where
		\[ v_{\mu} = \begin{pmatrix}
			a_1 \\
			\vdots \\
			a_s \\
			b_1 \\
			\vdots \\
			b_s \\
		\end{pmatrix}, \quad w_{\mu} = Hv_{\mu} = 
		\begin{pmatrix}
			a_1 \\
			\vdots \\
			a_s \\
			-iz_1b_1 \\
			\vdots \\
			-iz_sb_s \\
		\end{pmatrix} \ . 
		\]
	\end{corollary}
	
	\begin{definition}
	    For $N \in \mathbb{N}$ and $y \in \mathbb{C}^{2N+1}$, let the norm
	    \[ \|y\|_{2,N}^2 := \frac{1}{2N}\sum_{k=0}^{2N}|y_k|^2 \ .\]
	\end{definition}
	
	\subsubsection{Proof of the upper bound}
	As in \cite{Li&Liao2020}, we choose any $\varphi$ such that $\varphi_y \in \{\nu:\|\tilde{\mathcal{F}}_{2N}x_{\nu}-y\|_{2,N} \leq \varepsilon\}$. Note that $x_{\mu}$ satisfies the same constraint $\|\tilde{\mathcal{F}}_{2N}x_{\mu}-y\|_{2,N} \leq \varepsilon$, which means that such $\varphi_y$ exists. Then we have:
	\[ \mathcal{E}(\mathcal{R},N,\varepsilon) \leq \sup_{\mu \in \mathcal{R}}\sup_{y\in B^N_{\varepsilon}(\mu)}\|\varphi_y - \mu\|_2.\]
	By Lemma 4.7 in \cite{Vandermonde} there exists $\Delta_0(M)$ such that for all $\Delta \leq \Delta_0$
	and any $x_{\mu} \in \mathbb{C}^{2G}_{\mathcal{R}(\Delta,\rho,s,\ell,\tau)}$, we have
	\begin{align*}
		\boldsymbol{t}:=\supp(\varphi_y-\mu) &\in \mathcal{R}(\Delta,\rho',s',\ell',\tau')\\
		\Rightarrow x_{\varphi_y} - x_{\mu} &\in \mathbb{C}^{2G}_{\mathcal{R}(\Delta,\rho',s',\ell',\tau')}
	\end{align*}
	where $s'\leq 2s$, $\ell' \leq 2\ell$, $\tau' \geq 1$ and $\rho'=8sM\tau'\Delta$. In addition we have $\rho' \leq \frac{\rho}{2} \leq \frac{1}{4s^2} \leq \frac{1}{s'^2}$ and in particular, $4\tau'\Delta < min(\rho', \frac{1}{{s'}^2}) = \rho'$, therefore, by applying Theorem \ref{mainTheorem}, we obtain that for $K:=\frac{\tau'}{\pi s'}$ and all $N$ satisfying
	\[ \max\bigg(\frac{\pi s'}{2sM\tau'\Delta}=\frac{4\pi s'}{\rho'},4{s'}^3\bigg) \leq N \leq \frac{\pi s'}{\tau'\Delta}\]
    
	we have $K \leq \frac{1}{N\Delta} \leq (2sM)K$ and:
	\begin{align*}
		\frac{2\varepsilon}{\|\varphi_y - \mu\|_2} &\geq \frac{\bigg|\|\tilde{\mathcal{F}}_{2N}x_{\varphi_y}-y\|_{2,N} - \|\tilde{\mathcal{F}}_{2N}x_{\mu}-y\|_{2,N}\bigg|}{\|\varphi_y - \mu\|_2} \geq \frac{\|\tilde{\mathcal{F}}_{2N}(x_{\varphi_y}-x_{\mu})\|_{2,N}}{\|\varphi_y - \mu\|_2} \\
		&= \frac{\|P_{N}(\boldsymbol{t})(v_{\varphi_y}-v_{\mu})\|_2}{\|v_{\varphi_y}-v_{\mu}\|_2} = \frac{\|U_{N}(\boldsymbol{t})(H(\boldsymbol{t})v_{\varphi_y}-H(\boldsymbol{t})v_{\mu})\|_2}{\|H^{-1}(\boldsymbol{t})(H(\boldsymbol{t})v_{\varphi_y}-H(\boldsymbol{t})v_{\mu})\|_2} \\ &\geq  \frac{\|U_{N}(\boldsymbol{t})(w_{\varphi_y}-w_{\mu})\|_2}{\|H^{-1}(\boldsymbol{t})\|_2\|w_{\varphi_y}-w_{\mu}\|_2} 
		\geq \sigma_{\min}(U_{N}(\boldsymbol{t}))  \\ &\geq C_{s',\ell'}{(N\Delta)}^{2\ell' -1}  \geq C_{2s,2\ell}{(N\Delta)}^{4\ell -1}.\qed
	\end{align*}

	\subsubsection{Proof of the lower bound}
	Pick any $(\Delta,\rho,2s,2\ell,\tau)$-clustered configuration $\boldsymbol{t}=(t_1,\dots,t_{2s})$. Let $w\in\mathbb{C}^{4s}$ be a unit norm singular vector of  $U_N(\boldsymbol{t})$ that corresponds to its smallest singular value, put $v=H^{-1}(\boldsymbol{t})w$ and define the corresponding $\mu=\sum_{j=1}^{2s} (v)_{2j-1} \delta_{t_j} + (v)_{2j} \delta'_{t_j} \in \mathcal{R}(\Delta,\rho,2s,2\ell,\tau)$ (so in fact according to our previous notation $w=w_{\mu}$ and $v=v_{\mu}$). By this  construction, we obtain
	\[ \sigma:= \sigma_{\min}(U_{N}(\supp(\mu))= \|U_{N}w_{\mu}\|_2 = \|P_{N}v_{\mu}\|_2 = \|\tilde{\mathcal{F}}_{2N}x_{\mu}\|_{2,N} \ . \]
	
	Write $\boldsymbol{t}$ as a disjoint union of two  $(\Delta,\rho,s,\ell,\tau)$-clustered configurations $\boldsymbol{t}_1,\boldsymbol{t}_2$, implying that $\mu=\mu_1-\mu_2$ where $\supp\mu_i = \boldsymbol{t}_i$ and $\mu_i\in{\mathcal{R}}(\Delta,\rho,s,\ell,\tau)$ for $i=1,2$.   Let $x_i=\frac{\varepsilon}{\sigma}x_{\mu_i} \in \mathbb{C}^{2G}_{\mathcal{R}(\Delta,\rho,s,\ell,\tau)}$,  for $i=1,2$, so that $\frac{\varepsilon}{\sigma}x_{\mu} = x_1-x_2$.
	
	Now suppose we are given the data:
	\[ y = \tilde{\mathcal{F}}_{2N}x_1 = \tilde{\mathcal{F}}_{2N}x_2 + \tilde{\mathcal{F}}_{2N}(x_1-x_2).\]
	Let $e:=\frac{1}{\sqrt{2N}}\tilde{\mathcal{F}}_{2N}(x_1-x_2) \in \mathbb{C}^{2N+1}$.
	The previous equations imply:
	\[ \|e\|_2 = \|\tilde{\mathcal{F}}_{2N}(x_1-x_2)\|_{2,N} = \frac{\varepsilon}{\sigma}\|\tilde{\mathcal{F}}_{2N}x_{\mu}\|_{2,N} = \varepsilon.\]
	For an arbitrary $\varphi$ we have
	\begin{align*}
		\frac{\varepsilon}{\sigma} &= \frac{\varepsilon}{\sigma} \|w_{\mu}\|_2 = \frac{\varepsilon}{\sigma} \|x_{\mu}\|_2\\
		&= \|x_1-x_2\|_2 \\
		& \leq \|x_1-x_{\varphi_{y}}\|_2 + \|x_2-x_{\varphi_{y}}\|_2\\
		& \leq 2\max_{k=1,2} \|x_k-x_{\varphi_{y}}\|_2
	\end{align*}
	and so by definition of $\mathcal{E}$ and Theorem \ref{UpperTheorem} we conclude that for $SRF \geq 2$ it holds
	\[ \mathcal{E}(\mathcal{R},N,\varepsilon) \geq \inf_{\varphi \in \mathcal{A}}\max_{k=1,2}\|x_{\varphi_y} - x_k\|_2 \geq \frac{\varepsilon}{2\sigma} \geq \frac{\varepsilon}{2 C_{\tau,2\ell}{(N\Delta)}^{4\ell -1}}. \qed\]

	\section{Numerical experiments}\label{sec:numerics}
	In order to validate the bounds of Theorems \ref{mainTheorem}
	and \ref{UpperTheorem}, we computed $\sigma_{\min}(U_N)$ for varying values of $\Delta, N, \ell, s$  and the actual clustering configurations. As before, we put $SRF:=\frac{1}{N\Delta}$. We checked two clustering scenarios:
	\begin{enumerate}
		\item
		Figure \ref{subfig:C1} - A single equispaced cluster of size $\ell$ in $[-\frac{\pi}{2}, -\frac{\pi}{2}+\ell\Delta]$ with the rest of the nodes equally spaced and maximally separated in $(-\frac{\pi}{2}+\ell\Delta, \frac{\pi}{2}]$.
		\item
		Figure \ref{subfig:C2} - A multi-cluster configuration with the first equispaced cluster of size $\ell_1$ in $[-\frac{\pi}{2}, -\frac{\pi}{2}+\ell\Delta]$ and the second equispaced cluster of size $\ell_2$ in $[\frac{\pi}{2} - \ell_2\Delta, \frac{\pi}{2}]$  with the rest of the nodes equally spaced and maximally separated in $( -\frac{\pi}{2}+\ell_1\Delta, \frac{\pi}{2} - \ell_2\Delta)$.
	\end{enumerate}
	\begin{figure}[H]
		\centering
		\begin{subfigure}[b]{0.8\linewidth}
			\includegraphics[width=\linewidth]{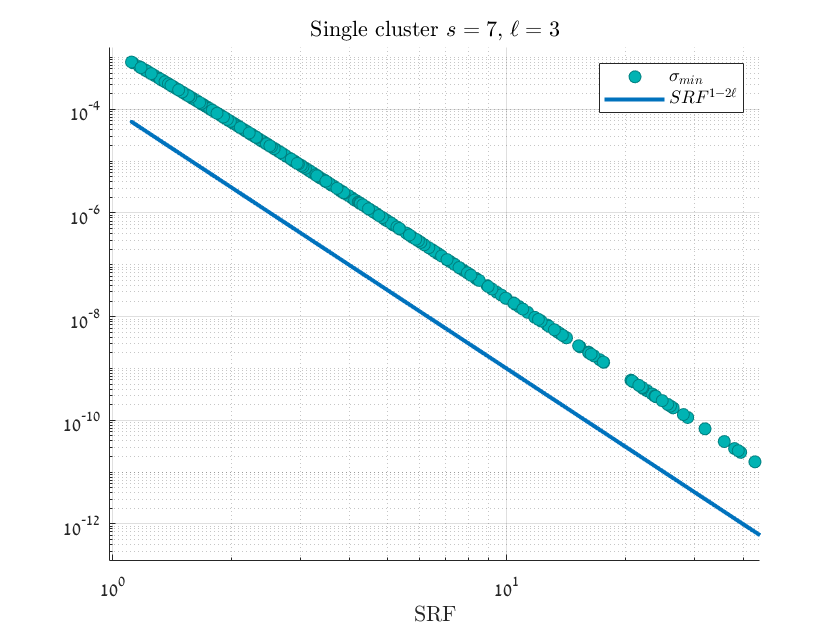}
			\caption{Single Cluster}
			\label{subfig:C1}
		\end{subfigure}
		\begin{subfigure}[b]{0.8\linewidth}
			\includegraphics[width=\linewidth]{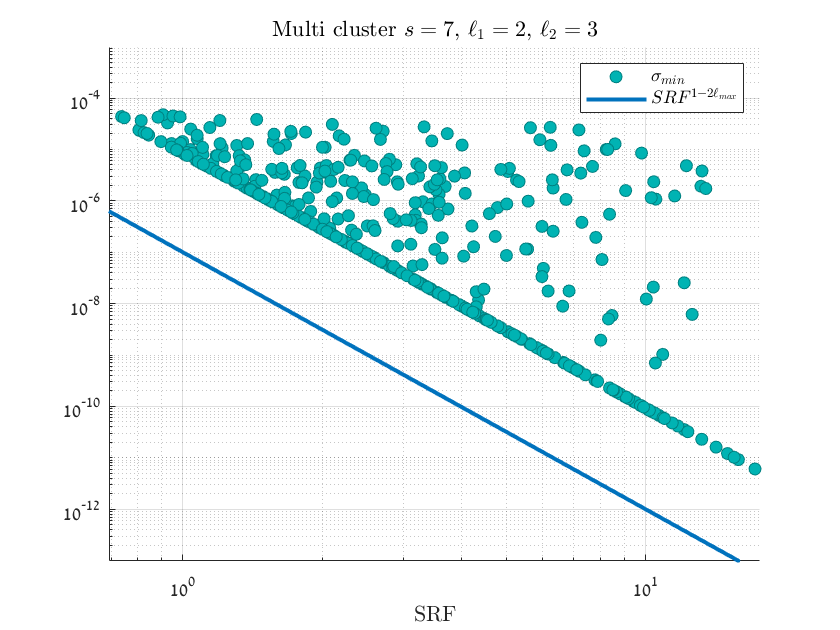}
			\caption{Multi Cluster, $\ell = \ell_{max}$}
			\label{subfig:C2}
		\end{subfigure}
		\caption{Decay rate of $\sigma_{\min}$ as a function of $SRF$. Results of $n=1000$ random experiments with randomly chosen $\Delta, N$ are plotted versus the theoretical bound ${SRF}^{1-2\ell}$.}
    \end{figure}
	We also show in figure \ref{fig:u} that the vector $u$ defined in \eqref{def:u} is indeed an approximate minimal singular vector,  by plotting the Rayleigh quotient $\frac{\|\Phi_M(\boldsymbol{\omega})u\|_2}{\|u\|_2}$ versus the minimal singular value $\sigma_{\min}(\Phi_M(\boldsymbol{\omega}))$, where $\Phi_M$ is the confluent Vandermonde matrix as in Definition \ref{unnormlizedConf} and $\boldsymbol{\omega}$ is a single-cluster equispaced configuration. 
	
	\begin{figure}[ht!]
		\includegraphics[width=\linewidth]{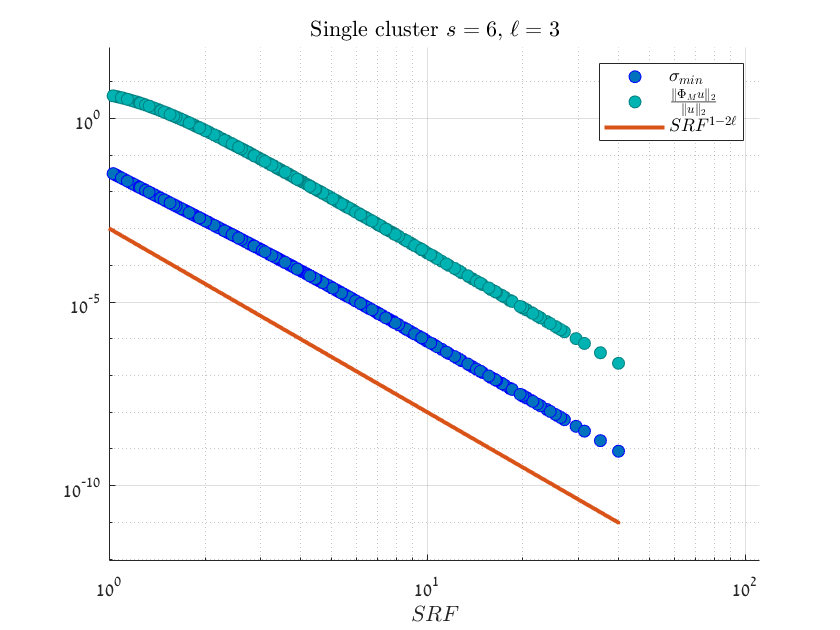}
		\caption{The Rayleigh quotient of the vector $u$ defined in \eqref{def:u} versus the minimal singular value of $\Phi_M$. We can see that they scale the same and differ by a constant.}
		\label{fig:u}
	\end{figure}

	Finally, in order to validate the bounds of Theorem \ref{theorem:min-max-error}, we computed the $\ell^2$ min-max error $\mathcal{E}$ as in Definition  \ref{def:minmax} and also the $\ell^2$ errors of estimating the nodes $\mathcal{E}_{\xi}$, and the coefficients $\mathcal{E}_{a}$, $\mathcal{E}_{b}$ of the worst-case discrete distribution $\mu$ defined by \eqref{eq:I} assuming $s=\ell$. We used the ESPRIT (Estimation of Signal Parameters
	via Rotation Invariance Techniques) \cite{ESPRIT1990} method for recovering the nodes $\{t_j\}_{j=1}^{s}$ (see more about this method in appendix \ref{appendix:ESPRIT}). ESPRIT is considered to be one of the best performing subspace methods for estimating parameters of model \eqref{simpleModel} with white Gaussian noise. Originally developed in the context of frequency estimation \cite{SpectralAnalysis2005}, it has been generalized to the full model \eqref{extendedmodel} in \cite{Badeau2006}. Recently it has been shown that if the noise level $\varepsilon$ in the measurements \eqref{eq:noisy-data} is sufficiently small, the error committed by ESPRIT  for estimating the nodes of the simple model \eqref{simpleModel} is nearly min-max \cite{Li2020}. Consequently, we conjecture the same near-optimal behaviour in the model \eqref{generalizedModel}. In order to recover the coefficients $a_j$, $b_j$ we solve a linear system of equations by the Least Squares method:
	\[ \min\|U_N(\tilde{\boldsymbol{\xi}})v_{\mu} - y\|_2 \ , \]
	where $\boldsymbol{\tilde{\xi}}$ are the recovered nodes.
	Note that we prove the theoretical bound to the on-grid model however the ESPRIT algorithm recovers the nodes without taking the grid assumption into account. 
	We have checked two cases:
	\begin{enumerate}
		\item
		Figure \ref{subfig:mm1} - A single equispaced cluster of size $s=\ell=2$ with error $\varepsilon = 10^{-12}$.
		\item
		Figure \ref{subfig:mm2} - A single equispaced cluster of size $s=\ell=3$ with error $\varepsilon = 10^{-12}$.
	\end{enumerate}
	Our results suggest that the ESPRIT method might indeed be optimal, meaning that it attains the min-max error bounds we established in Theorem \ref{theorem:min-max-error} for the recovered parameters of signal \eqref{generalizedModel}.
	\begin{figure}[h]
		\centering
		\begin{subfigure}[c]{1\linewidth}
			\includegraphics[width=\linewidth]{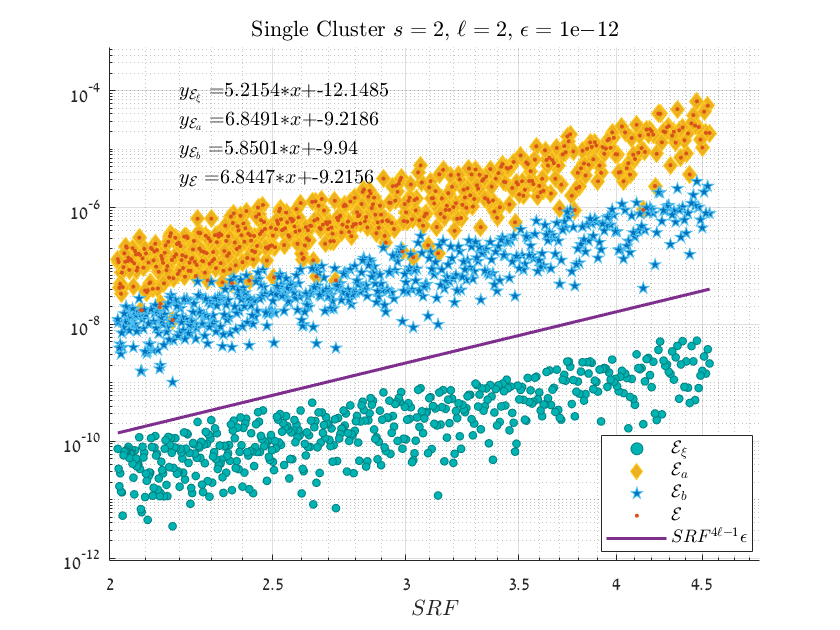}
			\caption{ }
			\label{subfig:mm1}
		\end{subfigure}
	\end{figure}
	\begin{figure}[H]\ContinuedFloat
	    \centering
		\begin{subfigure}{1\linewidth}
			\includegraphics[width=\linewidth]{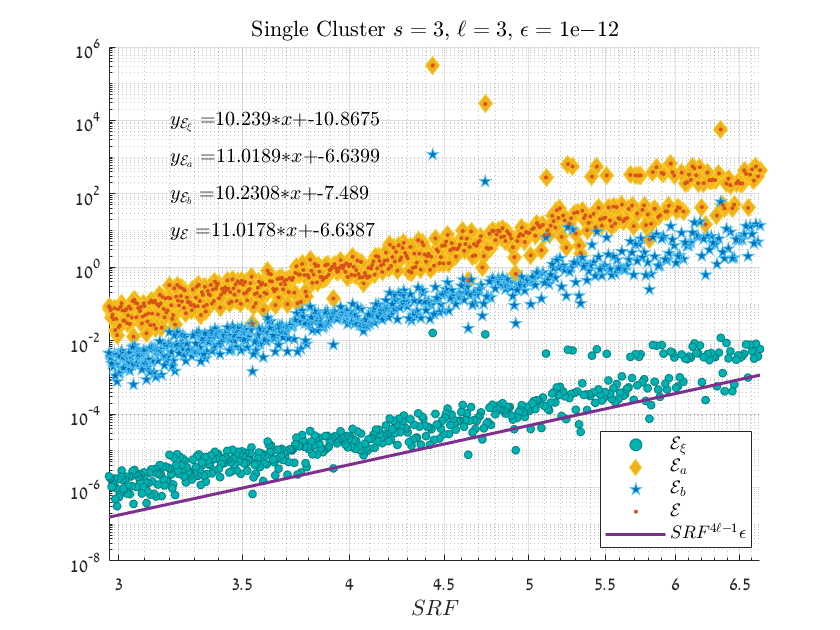}
			\caption{ }
			\label{subfig:mm2}
		\end{subfigure}
		\caption{Accuracy of ESPRIT. Results of n=500 random experiments with randomly chosen $\Delta$ and fixed $N$ are plotted versus the theoretical bound ${SRF}^{4\ell-1}\varepsilon$.}
	\end{figure}
	Note that all figures are in logarithmic scale.
	
	The code for the above experiments is available at  \url{https://github.com/Gnflu/SR-of-conVan-sys.git}
	
	\appendix
	\section{Computations for Theorem \ref{UpperTheorem}}
	\subsection{Finite difference coefficients}\label{appendix:taylor}	
	We seek approximation of the form:
	\[ D_M^{(2\ell -1)}(\omega) \approx \sum_{i=1}^{\ell}A_iD_M(x_i)+B_iD'_M(x_i) = S_{A,B}(\omega)\]
	where \[ x_i = \omega-\tau_i\Delta, \quad \Delta=\frac{\alpha}{M} \ , \]
	and
	\[ S_{A,B}(\omega)=\sum_{i=1}^{\ell}A_iD_M(x_i)+B_iD'_M(x_i) = \sum_{i=1}^{\ell}A_iD_M(\omega+x_i-\omega)+B_iD'_M(\omega+x_i-\omega).\]
	Let $h_i := x_i -\omega = -\tau_i\Delta = \tau_ih$, $h=-\Delta$. Then by Taylor expansion of $D_M(\omega+h_i)$ and using the integral form of the remainder we have:
	\begin{align*}
		S_{A,B}(\omega) = \sum_{i=1}^{\ell}A_i\bigg(\sum_{k=0}^{2\ell -1}\frac{D^{(k)}_M(\omega)}{k!}h_i^k + \int_{\omega}^{\omega+h_i}\frac{D_M^{(2\ell)}(t)}{(2\ell -1)!}(\omega+h_i-t)^{2\ell -1}\,dt \bigg) \\
		+B_i\bigg(\sum_{k=0}^{2\ell -2}\frac{D^{(k+1)}_M(\omega)}{k!}h_i^k + \int_{\omega}^{\omega+h_i}\frac{D_M^{(2\ell)}(t)}{(2\ell -2 )!}(\omega+h_i-t)^{2\ell -2}\,dt \bigg) 
	\end{align*}
	By the change of variable $t=\omega+h_ir$ we have $dt=h_idr$ and therefore
	\begin{align*}
		S_{A,B}(\omega) &= \sum_{k=0}^{2\ell -1}\frac{D_M^{(k)}(\omega)}{k!}\big(\sum_{i=1}^{\ell}A_ih_i^k \big)
		+ \sum_{k=0}^{2\ell -2}\frac{D_M^{(k+1)}(\omega)}{k!}\big(\sum_{i=1}^{\ell}B_ih_i^k \big) \\
		& \qquad + \frac{1}{(2\ell -1)!}\sum_{i=1}^{\ell}A_i\int_{0}^{1}D_M^{(2\ell)}(\omega+h_ir)(h_i-h_ir)^{2\ell -1}h_i\,dr \\
		&\qquad + \frac{1}{(2\ell -2)!}\sum_{i=1}^{\ell}B_i\int_{0}^{1}D_M^{(2\ell)}(\omega+h_ir)(h_i-h_ir)^{2\ell -2}h_i\,dr \\ &= P(\omega) + R_A(\omega) + R_B(\omega),
	\end{align*}
	where
	\begin{align*}
		P(\omega) &= \sum_{k=0}^{2\ell -1}\frac{D_M^{(k)}(\omega)}{k!}\big(\sum_{i=1}^{\ell}A_ih_i^k \big)
		+ \sum_{k=0}^{2\ell -2}\frac{D_M^{(k+1)}(\omega)}{k!}\big(\sum_{i=1}^{\ell}B_ih_i^k \big)\ , \\
		R_A(\omega) &= \frac{1}{(2\ell -1)!}\sum_{i=1}^{\ell}A_i\int_{0}^{1}D_M^{(2\ell)}(\omega+h_ir)(h_i-h_ir)^{2\ell -1}h_i\,dr \ , \\
		R_B(\omega) &= \frac{1}{(2\ell -2)!}\sum_{i=1}^{\ell}B_i\int_{0}^{1}D_M^{(2\ell)}(\omega+h_ir)(h_i-h_ir)^{2\ell -2}h_i\,dr .
	\end{align*}
	We seek $A_1,\dots,A_{\ell}$ and $B_1,\dots,B_{\ell}$ so that $P(\omega)\equiv D_M^{(2\ell -1)}(\omega)$, thus the following equations should be fulfilled:
	
	\begin{flushleft}
		\begin{enumerate}
			\item	
			\[\sum_{i=1}^{\ell}A_i = 0\] 
			\item
			\[\sum_{i=1}^{\ell}(A_i+\frac{kB_i}{h_i})h_i^k = 0 \quad \quad k = 1,\dots,2\ell -2\]
			\item
			\[\sum_{i=1}^{\ell}(A_i+\frac{(2\ell -1)B_i}{h_i})h_i^{2\ell -1} = (2\ell -1)! \]
		\end{enumerate}
	\end{flushleft}
	This is equivalent to solving the following linear system of equations:
	\[	
	U_{2\ell}\begin{pmatrix}
		A_1 \\
		\vdots \\
		A_{\ell} \\
		B_1 \\
		\vdots \\
		B_{\ell} \\
	\end{pmatrix}
	= \begin{pmatrix}
		0 \\
		\vdots \\
		\vdots \\
		0 \\
		(2\ell -1)! \\
	\end{pmatrix}
	\ , \]
	where
	\[ U_{2\ell} = \begin{pmatrix}
		1 & \dots & 1 & 0 & \dots & 0\\
		h_1 & \dots & h_{\ell} & 1 & \dots & 1\\
		\vdots & \dots & \vdots & \vdots & \dots & \vdots\\
		h_1^{2\ell -1} & \dots & h_{\ell}^{2\ell -1} & (2\ell -1)h_1^{2\ell -2} & \dots & (2\ell -1)h_{\ell}^{2\ell -2}
	\end{pmatrix} \ .
	\]
	Thus $A_j, B_j$ are given by:
	\begin{equation}\label{eq:AB} \tag{$\star\star$}
		\begin{pmatrix}
			A_1 \\
			\vdots \\
			A_{\ell} \\
			B_1 \\
			\vdots \\
			B_{\ell}
		\end{pmatrix} =
		U^{-1}_{2\ell}
		\begin{pmatrix}
			0 \\
			\vdots \\
			\vdots \\
			0 \\
			(2\ell -1)!
		\end{pmatrix}.
	\end{equation}
	In particular, if $U^{-1}_{2\ell} = \begin{pmatrix}
		V \\ W
	\end{pmatrix}$
	then $A_j = (2\ell -1)! v_{j,2\ell}$ and $B_j = (2\ell -1)! w_{j,2\ell}$,
	where $V, W \in \mathbb{C}^{\ell \times 2\ell}$ and $v_{i,j}, \ w_{i,j}$ denote the $(i,j)th$ entry of $V, W$ respectively.
	\\
	\subsection{Proof of Lemma \ref{lemma:RaRb}}\label{appendix:RaRb}
	Let
	\[ h_*:= \argmax_{h_i}\bigg|\int_{0}^{1}D_M^{(2\ell)}(\omega+h_ir)(h_i-h_ir)^{2\ell -1}h_i\,dr\bigg| \]
	
	Using the Cauchy-Schwartz inequality we have
	\begin{align*}
		\big\|R_A(\omega) \big\|^2_{L^2(\mathbb{T})} &= \frac{1}{(2\ell -1)!^2}\bigg\|\sum_{i=1}^{\ell}A_i\int_{0}^{1}D_M^{(2\ell)}(\omega+h_ir)(h_i-h_ir)^{2\ell -1}h_i\,dr \bigg\|^2_{L^2(\mathbb{T})} \\
		&= \frac{1}{(2\ell -1)!^2}\int_0^{1}\bigg|\sum_{i=1}^{\ell}A_i\int_{0}^{1}D_M^{(2\ell)}(\omega+h_ir)(h_i-h_ir)^{2\ell -1}h_i\,dr\bigg|^2\,d\omega \\
		&\leq \frac{1}{(2\ell -1)!^2}\bigg(\sum_{i=1}^{\ell}\big|A_i\big|\bigg)^2\int_0^{1}\bigg|\int_{0}^{1}D_M^{(2\ell)}(\omega+h_*r)(h_*-h_*r)^{2\ell -1}h_*\,dr\bigg|^2\,d\omega \\
		&\leq \frac{1}{(2\ell -1)!^2}\bigg(\sum_{i=1}^{\ell}\big|A_i\big|\bigg)^2\int_0^{1}\bigg(\int_{0}^{1}\big|D_M^{(2\ell)}(\omega+h_*r)\big|^2\,dr\int_{0}^{1}\big|h_*-h_*r\big|^{4\ell -2}h_*^2\,dr\bigg)\,d\omega \\
		&\leq \frac{1}{(2\ell -1)!^2}\bigg(\sum_{i=1}^{\ell}\big|A_i\big|\bigg)^2\int_0^{1}\bigg(\int_{0}^{1}|D_M^{(2\ell)}(\omega+h_*r)|^2\,dr |\tau\Delta|^{4\ell}\bigg)\,d\omega \\
		&\leq \frac{1}{(2\ell -1)!^2}\bigg(\sum_{i=1}^{\ell}\big|A_i\big|\bigg)^2\big|\tau\Delta\big|^{4\ell}\int_0^{1}\bigg(\int_{0}^{1}|D_M^{(2\ell)}(\omega+h_*r)|^2\,d\omega\bigg) \,dr \\
		&\leq \frac{1}{(2\ell -1)!^2}\bigg(\sum_{i=1}^{\ell}\big|A_i\big|\bigg)^2\big|\tau\Delta\big|^{4\ell}\big\|D_M^{(2\ell)}\big\|^2_{L^2(\mathbb{T})} \\ &\leq \frac{\tau^{4\ell}}{(2\ell -1)!^2}\bigg(\sum_{i=1}^{\ell}\big|A_i\big|\bigg)^2\Delta^{4\ell}(M+1)(2\pi M)^{4\ell} 
	\end{align*}
	Similarly we get that:
	\[ \|R_B(\omega) \|^2_{L^2(\mathbb{T})} \leq \frac{\tau^{4\ell -2}}{(2\ell -1)!^2}\bigg(\sum_{i=1}^{\ell}\big|B_i\big|\bigg)^2\Delta^{4\ell -2}(M+1)(2\pi M)^{4\ell} \]
	
	\subsection{Proof of Lemma \ref{lemma:AB}}\label{appendix:AB}
	Recall the definitions of $v_{i,j}$ and $w_{i,j}$ from Appendix \ref{appendix:taylor}.
	From expressions (3.10) and (3.12) evaluated in Gautchi's paper \cite{Gautschi1963}, and using that $M \leq \frac{1}{\Delta}$ we have:
	\begin{enumerate}
		\item
		On one hand
		\begin{align*}
			\sum_{\mu=1}^{2\ell}\big|v_{n,\mu}\big| &\leq \Bigg(\bigg|1+2h_n\sum_{\nu \neq n}^{ }\frac{1}{(h_{n}-h_{\nu})}\bigg|+2\bigg|\sum_{\nu \neq n}^{ }\frac{1}{(h_{n}-h_{\nu})}\bigg|\Bigg)\prod_{\nu \neq n}^{ }\Big(\frac{1+|h_{\nu}|}{|h_{n}-h_{\nu}|}\Big)^2 \\
			&\leq \Bigg(\bigg|1+2\tau_nh\sum_{\nu \neq n}^{ }\frac{1}{(\tau_n-\tau_\nu)h}\bigg|+2\bigg|\sum_{\nu \neq n}^{ }\frac{1}{(\tau_n-\tau_\nu)h}\bigg|\Bigg)\prod_{\nu \neq n}^{ }\Big(\frac{1+|\tau_\nu h|}{|(\tau_n-\tau_\nu)h|}\Big)^2 \\
			&\leq \Big(\big|1+2{\tau(\ell-1)}\big|+\big|\frac{2\ell}{h}\big|\Big)\Big(\frac{1+\tau|h|}{|h|}\Big)^{2\ell-2} \\
			&= \Big(1 + 2\tau\ell -2\tau + 2\ell\frac{M}{\alpha}\Big)\Big(\frac{M}{\alpha}+\tau\Big)^{2\ell-2} \\
			&\leq 2\ell\Big(\frac{M}{\alpha}+\tau\Big)^{2\ell-1} \\
			&\leq C_v(\ell,\tau)\Big(\frac{M}{\alpha}\Big)^{2\ell-1}
		\end{align*}
		\begin{comment}
		where 
		\[ C_v(\ell)= 2\ell{\bigg(\frac{\ell +1}{\ell}\bigg)}^{2\ell -1}\]
		\end{comment}
		\item
		On the other hand,
		\begin{align*}
			\sum_{\mu=1}^{2\ell}\big|w_{n,\mu}\big| &\leq \big(1+|h_n|\big)\prod_{\nu \neq n}^{ }\Big(\frac{1+|h_{\nu}|}{|h_{n}-h_{\nu}|}\Big)^2 
			\leq \big(1+\tau|h|\big)\Big(\frac{1+\tau|h|}{|h|}\Big)^{2\ell-2} \\
			&= \Big(1+\tau\frac{\alpha}{M}\Big)\Big(\frac{M}{\alpha}+\tau\Big)^{2\ell-2} \leq 2\Big(\frac{M}{\alpha}+\tau\Big)^{2\ell-2} \\
			&\leq C_w(\ell,\tau)\Big(\frac{M}{\alpha}\Big)^{2\ell-2}
		\end{align*}
		\begin{comment}
		where 
		\[ C_w(\ell)= 2{\bigg(\frac{\ell +1}{\ell}\bigg)}^{2\ell -2}\]
		\end{comment}
		
	\end{enumerate}
	Now we can evaluate the following expressions:
	\begin{enumerate}
		\item
		\begin{align*}
			\sum_{i=0}^{\ell}|A_i| &= (2\ell-1)!\sum_{i=1}^{\ell}|v_{i,2\ell}| \leq (2\ell-1)!\sum_{i=1}^{\ell}\sum_{\mu=1}^{2\ell}|v_{i,\mu}| \\
			&\leq \ell(2\ell-1)!C_v(\ell,\tau)\bigg(\frac{M}{\alpha}\bigg)^{2\ell-1}\\ &= C_A(\ell,\tau)\bigg(\frac{M}{\alpha}\bigg)^{2\ell-1}
		\end{align*}
		\item
		\begin{align*}
			\sum_{i=0}^{\ell}|B_i| &= (2\ell-1)!\sum_{i=1}^{\ell}|w_{i,2\ell}| \leq (2\ell-1)!\sum_{i=1}^{\ell}\sum_{\mu=1}^{2\ell}|w_{i,\mu}| \\
			&\leq \ell(2\ell-1)!C_w(\ell,\tau)\bigg(\frac{M}{\alpha}\bigg)^{2\ell-2}\\ &= C_B(\ell,\tau)\bigg(\frac{M}{\alpha}\bigg)^{2\ell-2}
		\end{align*}
	\end{enumerate}
	\subsection{Proof of Lemma \ref{lemma:L2}}\label{appendix:L2}
	For any tempered distribution $\mu$ supported in $\mathbb{T}$, we will show that the following is true:
	\[\sum_{m=0}^{M}\big|\hat{\mu}(m)\big|^2 = \big\|\mu * D_M\big\|^2_{L^2(\mathbb{T})}\]
	
	First:
	\begin{align*}
		(\mu * D_M)(\omega) &= \int_{0}^{1}\mu(y)D_m(\omega -y) \, dy = \int_{0}^{1}\mu(y)\sum_{m=0}^{M}e^{2\pi im(\omega-y)} \, dy \\ &=
		\sum_{m=0}^{M}e^{2\pi im\omega}\bigg(\int_{0}^{1}\mu(y)e^{-2\pi imy} \, dy \bigg) \\ &= \sum_{m=0}^{M}\hat{\mu}(m)e^{2\pi im\omega} \ .
	\end{align*}

	Now we can show the desired equality:
	\begin{align*}
		\big\|(\mu * D_M)(\omega)\big\|^2_{L^2(\mathbb{T})}
		&= \int_{0}^{1}\big|\sum_{m=0}^{M}e^{2\pi im\omega}\hat{\mu}(m)\big|^2 \, d\omega = \int_{0}^{1}\big(\sum_{m=0}^{M}e^{2\pi im\omega}\hat{\mu}(m)\big)\overline{\big(\sum_{m=0}^{M}e^{2\pi im\omega}\hat{\mu}(m)\big)} \, d\omega \\
		&= \int_{0}^{1}\sum_{m=0}^{M}\big|e^{2\pi im\omega}\hat{\mu}(m)\big|^2 + \sum_{m=0}^{M}\sum_{k \neq m}^{}e^{2\pi i(m-k)\omega}\hat{\mu}(m)\overline{\hat{\mu}(k)} \, d\omega \\
		&= \sum_{m=0}^{M}\big|\hat{\mu}(m)\big|^2 + \sum_{m=0}^{M}\sum_{k \neq m}^{}\hat{\mu}(m)\overline{\hat{\mu}(k)}\int_{0}^{1}e^{2\pi i(m-k)\omega} \, d\omega \\ &= \sum_{m=0}^{M}\big|\hat{\mu}(m)\big|^2 \qed
	\end{align*}
	\subsection{Proof of Lemma \ref{lemma:u}}\label{appendix:u}
	Let $U_{2\ell}$ be defined as in \eqref{appendix:taylor}, and let
	\[X = \begin{pmatrix}
		A_1 \\
		\vdots \\
		A_{\ell} \\
		B_1 \\
		\vdots \\
		B_{\ell} 
	\end{pmatrix}, \quad b = 
	\begin{pmatrix}
		0 \\
		\vdots \\
		\vdots \\
		0 \\
		(2\ell-1)! 
	\end{pmatrix}\]
	As in appendix \ref{appendix:taylor} we know that:
	\[ U_{2\ell}x = b \tag{*}\]
	We can write (*) as:
	\[
	\underbrace{\begin{pmatrix}
			1 & & & \\
			& h & & \\
			& & \ddots & \\
			& & & h^{2\ell-1}
	\end{pmatrix}}_{D}
	\underbrace{\begin{pmatrix}
			1 & \dots & 1 & 0 & \dots & 0\\
			\tau_1 & \dots & \tau_{\ell} & 1 & \dots & 1\\
			\vdots & \dots & \vdots & \vdots & \dots & \vdots\\
			\tau_1^{2\ell -1} & \dots & \tau_{\ell}^{2\ell -1} & (2\ell -1)\tau_{1}^{2\ell -2} & \dots & (2\ell -1)\tau_{\ell}^{2\ell -2}
	\end{pmatrix}}_{L}
	\underbrace{\begin{pmatrix}
			A_1 \\
			\vdots \\
			A_{\ell} \\
			\frac{B_1}{h} \\
			\vdots \\
			\frac{B_{\ell}}{h} 
	\end{pmatrix}}_{y}
	= b
	\]
	\[ y = L^{-1}D^{-1}b, \quad x=\diag{1,\dots,1,h,\dots,h}y\]
	\[ x= \begin{pmatrix}
		I & 0 \\
		0 & hI
	\end{pmatrix}L^{-1}D^{-1}b = \begin{pmatrix}
		I & 0 \\
		0 & hI
	\end{pmatrix}L^{-1}\begin{pmatrix}
		0 \\ 					
		\vdots \\
		0 \\
		(2\ell -1)!h^{1-2\ell} 
	\end{pmatrix}
	= \begin{pmatrix}
		{(L^{-1})}_{1,2\ell}(2\ell -1)!h^{1-2\ell} \\
		\vdots \\
		{(L^{-1})}_{\ell,2\ell}(2\ell -1)!h^{1-2\ell} \\
		{(L^{-1})}_{\ell+1,2\ell}(2\ell -1)!h^{2-2\ell} \\
		\vdots \\
		{(L^{-1})}_{2\ell,2\ell}(2\ell -1)!h^{2-2\ell} 
	\end{pmatrix}
	\]
	\[ \|x\|^2_2 = \sum_{i=1}^{\ell}{(L^{-1})}_{i,2\ell}^2{(2\ell -1)!}^2h^{2(1-2\ell)} + \sum_{i=1}^{\ell}{(L^{-1})}_{\ell+i,2\ell}^2{(2\ell -1)!}^2h^{2(2-2\ell)} \geq \sum_{i=1}^{2\ell}{(L^{-1})}_{i,2\ell}^2{(2\ell -1)!}^2\Delta^{2-4\ell}\]
	\begin{align*}
		\sum_{i=1}^{2\ell}{(L^{-1})}_{i,2\ell}^2 = \|L^{-1}(0,...,0,1)^T\|_2^2 \geq \min_{\|w\|_2=1}\|L^{-1}w\|_2^2 = \sigma_{\min}^2(L^{-1})  =  \frac{1}{\sigma_{max}^2(L)} = \frac{1}{\|L\|_2^{2}} \geq \frac{1}{rank(L)\|L\|_{\infty}^{2}}
	\end{align*}
	\begin{align*}
		\|L\|_{\infty} \leq \sum_{i=1}^{\ell}\tau_i^{2\ell -1}  + \sum_{i=1}^{\ell}(2\ell -1)\tau_i^{2\ell -2} \leq \ell(\tau^{2\ell -1}+(2\ell -1)\tau^{2\ell -2}) \leq 2\ell^2\tau^{2\ell -1}
	\end{align*}
	\begin{align*}
		\|x\|_2^2 \geq \frac{{(2\ell -1)!}^2\Delta^{2-4\ell}}{2\ell {(2\ell^2\tau^{2\ell -1})}^2}  \Rightarrow
		\|x\|_2 \geq \frac{(2\ell -1)!\Delta^{1-2\ell}}{\sqrt{8\ell^5}\tau^{2\ell -1}} = C_3(\ell,\tau)\Delta^{1-2\ell}
	\end{align*}
	\begin{comment}
	\|x\|_2 \geq (C_1^2(\ell)\Delta^{2(1-2\ell)} + C_2^2(\ell)\Delta^{2(2-2\ell)})^{\frac{1}{2}} = \Delta^{1-2\ell}(C_1^2(\ell)+C_2^2(\ell)\Delta^2)^{\frac{1}{2}} \\
	= \Delta^{1-2\ell}(C_1^2(\ell)+C_2^2(\ell)\Delta^2 + 2C_1(\ell)C_2(\ell)\Delta - 2C_1(\ell)C_2(\ell)\Delta)^{\frac{1}{2}} \\
	\geq  \Delta^{1-2\ell}(C_1^2(\ell)+C_2^2(\ell)\Delta^2 + 2C_1(\ell)C_2(\ell)\Delta)^{\frac{1}{2}}\\
	= \Delta^{1-2\ell}(C_1(\ell)+C_2(\ell)\Delta) 
	\end{comment}
	Finally, for $u$ defined in \eqref{def:u}, we have
	\[ u = \Delta^{2\ell-1}\begin{pmatrix}
		I & 0 \\
		0 & \mathbf{Z}
	\end{pmatrix}X=\Delta^{2\ell-1}\mathbf{Z}x, \quad \mathbf{Z}=\diag{z_1,\dots,z_{\ell}}
	\]
	\begin{align*}
		\|u\|_2 &= \Delta^{2\ell-1}\|\mathbf{Z}x\|_2\|\mathbf{Z}^{-1}\|_2\|\mathbf{Z}^{-1}\|^{-1}_2 
		\\ &\geq \Delta^{2\ell-1}\|\mathbf{Z}\mathbf{Z}^{-1}x\|_2\|\mathbf{Z}^{-1}\|^{-1}_2  = \Delta^{2\ell-1}\|x\|_2\|\mathbf{Z}^{-1}\|^{-1}_2 \\
		&\geq \frac{1}{\sqrt{2\ell}}\Delta^{2\ell-1}\Delta^{1-2\ell}C_3(\ell,\tau) = \tilde{C}_3(\ell,\tau)
	\end{align*}
	We used in last inequality the following property:
	\[ \|\mathbf{Z}\|_2 \leq \sqrt{rank(\mathbf{Z})}\|\mathbf{Z}\|_{\infty}\]
	
	\section{ESPRIT Method}\label{appendix:ESPRIT}
	We provide the description of the matrix for completeness, see e.g. \cite{Badeau2006,batenkov2013c}.
	\begin{definition}[Hankel Matrix]\label{Hankel}
		Let $\mu(x)=\sum_{j=1}^{s}a_j\delta(x-t_j)+b_j\delta'(x-t_j)$, $a_j, b_j \in \mathbb{C}$ and $\boldsymbol{t}=(t_1,\dots,t_s)$, $t_j \in \mathbb{T}$, thus
		\[ m_k := \hat{\mu}(k) = \sum_{j=1}^{s}a_je^{-2\pi ikt_j}+2\pi ikb_je^{-2\pi ikt_j} = \sum_{j=1}^{s}a_jz_j^k+2\pi ikb_jz_j^k\]
		Then we define the $C\times C$ Hankel matrix as follows:
		\[ H_C:= \begin{pmatrix}
			m_0 & m_1 & \dots & m_{C-1} \\
			m_1 & m_2 & \dots & m_{C} \\
			\vdots & \vdots & \vdots & \vdots \\
			m_{C-1} & m_C & \dots & m_{2C-2}
		\end{pmatrix}\]
		where $C:=2s$ (number of unknown coeffients).
	\end{definition}
	
	The ESPRIT (and other subspace methods) relies on the following
	observations:
	\begin{enumerate}
		\item
		The range (column space) of both the data matrix $H_C$ \eqref{Hankel} and the confluent
		Vandermonde matrix $\Phi:=\Phi_{2C-1}$ \refeq{reconfluentVan} are the same, namely $H_C$ admits the following factorization: 
		\[H_C = \Phi B\Phi^T \ , \] where $B:=\diag{a_1,\dots,a_s,b_1,\dots,b_s}$.
		\item 
		The matrix $\Phi$ has the so-called rotational invariance property \cite{Badeau2006}:
		\[ \Phi^{\uparrow} = \Phi^{\downarrow}J\]
		where $\Phi^{\uparrow}$ denotes $\Phi$ without the first row, $\Phi^{\downarrow}$ denotes $\Phi$ without the last
		row, and $J$ is a block diagonal matrix whose $i^{\text{th}}$ block is the $2\times 2$ Jordan
		block with the node $z_i$ on the diagonal.
	\end{enumerate}
	Suppose we know $\Phi$; then the matrix $J$ could be found by
	\[ J = \Phi_{\downarrow}^{\#}\Phi^{\uparrow}\]
	(where $\#$ denotes the Moore–Penrose pseudoinverse), and then the nodes $z_j$ could be recovered as the eigenvalues of $J$.
	\par Unfortunately, $\Phi$ is unknown in advance, but suppose we had at our disposal a matrix $W$ whose column space was identical to that of $\Phi$. In that case, we would
	have $W = \Phi G$ for an invertible $G$, and consequently
	\[ W^{\uparrow} = W^{\downarrow}\Psi, \]
	where
	\[ \Psi = G^{-1}JG,\]
	which means that the eigenvalues of $\Psi$ are also $\{z_j\}$. Such a matrix $W$ can be obtained, for example, from the singular value decomposition (SVD) of the data matrix/covariance matrix. To summarize, the ESPRIT method for estimating $\{z_j\}$, as used in our experiments below, is as follows.
	\begin{algorithm}
		\begin{algorithmic}[1]
			\Require A $C\times C$ Hankel matrix $H_C$ built from the measurements.
			\Ensure Recovered nodes $\{z_j\}$.
			\State Compute the SVD $H_C=W\Sigma V^T$ 
			\State Calculate $\Psi=W^{\#}_{\downarrow}W^{\uparrow}$ 
			\State Set $\{z_j\}$ to be the eigenvalues of $\Psi$ with appropriate multiplicities (use, e.g., arithmetic neans to estimate multiple nodes which are scattered by the noise).
		\end{algorithmic}
		\caption{ESPRIT method for recovering the nodes $\{z_j\}$}
		\label{alg:ESPRIT}
	\end{algorithm}
	
	\printbibliography
	
\end{document}